\documentclass[11pt]{epiarticle}
\usepackage{epimath}
\usepackage{cancel}
\usepackage{pgf,tikz}
\usetikzlibrary{shapes,arrows,shadows,patterns,shapes.geometric,calc}
\usepackage{graphicx}

\setpapertype{A4}

\usepackage[english]{babel}

\newtheorem{definition}{Definition}

\newtheorem{query}{Query}

\title{\textit{Why} the classes P and NP are not well-defined \textit{finitarily}\\ \vspace{+1ex} \large{G\"{o}del's `formally' unprovable, but `numeral-wise' provable, arithmetical formula $[R(x)]$ is \textit{not} algorithmically \textit{computable} as a tautology by \textit{any} Turing machine, whether \textit{deterministic} or \textit{non-deterministic}}}

\titlemark{P$v$NP}
\author{\normalsize{Bhupinder Singh Anand}}
\authormark{B S Anand}
\date{\footnotesize{Mumbai, India} \\ \footnotesize{bhup.anand@gmail.com} \tiny({\href{https://orcid.org/0000-0003-4290-9549}{https://orcid.org/0000-0003-4290-9549}})} % Empty date or tweak it according to your needs
\journal{arXived on \today} % Epijournal name
\begin{document}
	
	%%%%%%%%%%%%%%%%%%%%%%%%%%%%%%%
	% Add the title to the document
	%%%%%%%%%%%%%%%%%%%%%%%%%%%%%%%
	\maketitle
	
	%%%%%%%%%%%%%%%%%%%%%
	% Dedication (if any)
	%%%%%%%%%%%%%%%%%%%%%
	% \dedication{in the memory of William V.D.\ Hodge}
	
	%%%%%%%%%%%%%%%%%%%%%%%%%%%%%%%%%%%%%%%%%%%%%%%%%%%%%%%%%%
	% Add abstract, Keywords, MSC classification (recommended)
	% Never remove prelims section, make it rather empty
	%%%%%%%%%%%%%%%%%%%%%%%%%%%%%%%%%%%%%%%%%%%%%%%%%%%%%%%%%%
	\begin{prelims}
		
		\def\abstractname{Abstract}
		\abstract{We distinguish \textit{finitarily} between algorithmic \textit{verifiability}, and algorithmic \textit{computability}, to show that G\"{o}del's `formally' unprovable, but `numeral-wise' provable, arithmetical proposition $[(\forall x)R(x)]$ can be \textit{finitarily} evidenced as: algorithmically \textit{verifiable} as `always' true, but not algorithmically \textit{computable} as `always' true. Hence, though $[R(x)]$ is algorithmically \textit{verifiable} as a tautology, it is not algorithmically \textit{computable} as a tautology by any Turing machine, whether \textit{deterministic} or \textit{non-deterministic}. By interpreting the P$v$NP problem arithmetically, rather than  set-theoretically, we conclude that the classes P and NP are not well-defined \textit{finitarily} since it immediately follows that SAT $\notin$ P, \textit{and} SAT $\notin$ NP.} \\  \\ \MSCclass 00A30, 03A05, 03B10, 03D10, 03D15, 03F25, 03F30, 68Q15, 68Q17, 68T27  \\ \tiny \textbf{DECLARATIONS} $\bullet$ \textbf{Funding}: Not applicable $\bullet$ \textbf{Conflicts of interest/Competing interests}: Not applicable $\bullet$ \textbf{Availability of data and material}: Not applicable $\bullet$ \textbf{Code availability}: Not applicable $\bullet$ \textbf{Authors' contributions}: Not applicable
		
		% Add table of contents (optional)
		% \tableofcontents
		
	\end{prelims}
	
	%%%%%%%%%%%%%%%%%%%%%
	% Content begins here
	%%%%%%%%%%%%%%%%%%%%%

	\section{Introduction}
	\label{sec:intro.prov.thm.pvnp}
	In this investigation, we highlight a significant entailment for computational complexity of the distinction between algorithmic \textit{verifiability}, and algorithmic \textit{computability}, introduced in the 2012 Turing-centenary paper \cite{An12} presented at the AISB/IACAP World Congress in Birmingham, and the paper \cite{An16} published in the December 2016 issue of \textit{Cognitive Systems Research}:
	
	\begin{quote}
		\footnotesize
		``A number-theoretical relation $F(x)$ is algorithmically verifiable if, and only if, for any given natural number $n$, there is an algorithm $AL_{(F,\ n)}$ which can provide objective evidence$^{18}$ for deciding the truth/falsity of each proposition in the finite sequence $\{F(1), F(2), \ldots, F(n)\}$." \\ \textit{\tiny{\ldots Anand: \cite{An16}, \S 2, \textbf{Definition 1} (\textit{Algorithmic verifiability}).}}
		
		\vspace{+1ex}
		``A number theoretical relation $F(x)$ is algorithmically computable if, and only if, there is an algorithm $AL_{F}$ that can provide objective evidence for deciding the truth/falsity of each proposition in the denumerable sequence $\{F(1), F(2), \ldots\}$." \\ \textit{\tiny{\ldots Anand: \cite{An16}, \S 2, \textbf{Definition 2} (\textit{Algorithmic computability}).}}
		
		\vspace{+1ex}
		``We note that algorithmic computability implies the existence of an algorithm that can finitarily decide the truth/falsity of each proposition in a well-defined denumerable sequence of propositions,$^{19}$ whereas algorithmic verifiability does not imply the existence of an algorithm that can finitarily decide the truth/falsity of each proposition in a well-defined denumerable sequence of propositions.		
		
		\vspace{+1ex}
		\tiny{$^{18}$ cf. Murthy (1991): ``It is by now folklore \ldots that one can view the values of a simple functional language as specifying evidence for propositions in a constructive logic \ldots".
			
		\vspace{+1ex}
		$^{19}$ We note that the concept of `algorithmic computability' is essentially an expression of the more rigorously defined concept of `realizability' in Kleene (1952), p. 503.}"
		
		\vspace{+1ex}
		\textit{\tiny{\ldots Anand: \cite{An16}, \S 2, Defining algorithmic verifiability and algorithmic computability.}}
	\end{quote}
	
	We show that the \textit{finitary} proof of consistency for PA in \cite{An12}, and in \cite{An16}\footnote{\cite{An16}, \textbf{Theorem 6.7}: The axioms of PA are always algorithmically \textit{computable} as true under the interpretation $\mathcal{I}_{PA(\mathbb{N},\ SC)}$ and the rules of inference of PA preserve the properties of algorithmically \textit{computable} satisfaction/truth under $\mathcal{I}_{PA(\mathbb{N},\ SC)}$.}, is a \textit{strong}\footnote{Vis \`{a} vis the \textit{weak} (see \S \ref{sec:pa.axms.true.alg.ver}, Theorem \ref{sec:5.4.lem.5}) proof of consistency for PA in \cite{An16}, \textbf{Theorem 5.6}: The axioms of PA are always algorithmically \textit{verifiable} as true under the interpretation $\mathcal{I}_{PA(\mathbb{N},\ SV)}$ and the rules of inference of PA preserve the properties of algorithmically \textit{verifiable} satisfaction/truth under $\mathcal{I}_{PA(\mathbb{N},\ SV)}$.} proof of consistency (see \S \ref{sec:pa.axms.true.alg.cmp}, Theorem \ref{sec:6.lem.5}). Moreover, it entails a Provability Theorem\footnote{\cite{An16}, \S 7, \textbf{Theorem 7.1} (\textit{Provability Theorem for PA}): A PA formula $[F(x)]$ is PA-provable if, and only if, $[F(x)]$ is algorithmically \textit{computable} as always true in $\mathbb{N}$.} (see \S \ref{sec:6.3.a}, Theorem \ref{sec:6.3.thm.1}) for the first-order Peano Arithmetic PA that bridges the---hitherto only intuitively conjectured (see \S \ref{sec:6.3.a})---gap between arithmetic provability and Turing-computability.
	
	\vspace{+1ex}
	We then show how the Provability Theorem for PA further entails that Kurt G\"{o}del's `formally' unprovable, but `numeral-wise' provable, arithmetical proposition in \cite{Go31}---which is expressible as a PA formula of the form $[(\forall x)R(x)]$\footnote{G\"{o}del \textit{finitarily} defines, and refers to, the formulas $[(\forall x)R(x)]$ and $[R(x)]$ \textit{only} by their G\"{o}del-numbers `$17\ Gen\ r$' and `$r$', respectively, in his seminal 1931 paper \cite{Go31} (in eqns. \#13 and \#12, respectively, on p.25). His ensuing argument of \cite{Go31}, Theorem VI, then entails that:
		
	\begin{enumerate}
		\item[(i)] $[(\forall x)R(x)]$ is not `formally' provable in PA;
		
		\begin{quote}
			``$17\ Gen\ r$ is not $\kappa$-PROVABLE." \textit{\tiny{\ldots G\"{o}del: \cite{Go31}, p.25(1)}};
		\end{quote}
		
		\item[(ii)] for any \textit{specified} numeral $[n]$, the formula $[R(n)]$ is `numeral-wise' provable in PA.
		
		\begin{quote}
		 ``From this, we deduce, according to (15), $(n) Bew_{\kappa}[Sb(r \begin{array}{c} 17 \\ Z(n) \end{array})]$ \ldots" \textit{\tiny{\ldots G\"{o}del: \cite{Go31}, p.26(2)}}.
		\end{quote} 
	\end{enumerate}
	}---\textit{well-defines} (by \S \ref{sec:intro.prov.thm.pvnp}, Definition \ref{def:const.wd.nt.seq}) a quantifier-free PA formula $[R(x)]$ that is algorithmically \textit{verifiable} (by \S \ref{sec:intro.prov.thm.pvnp}, Definition \ref{sec:2.def.1}), but not algorithmically \textit{computable} (by \S \ref{sec:intro.prov.thm.pvnp}, Definition \ref{sec:2.def.2}), as a tautology.
	
	\vspace{+1ex}
	By interpreting the P$v$NP problem arithmetically, rather than set-theoretically (see \S \ref{sec:pvnp.problem}), we thus conclude that SAT \(\notin\) P \textit{and} SAT \(\notin\) NP (\S \ref{sec:lgcl.prf.p.neq.np}, Theorem \ref{thm:sat.prf.pvnp}), since \textit{no} Turing-machine, whether \textit{deterministic} or \textit{non-deterministic}, can \textit{evidence} whether or not $[R(x)]$ is a tautology\footnote{See also:
		
		\begin{enumerate}
			\item[(i)] \cite{An21}, \S 20.E, \textbf{Corollary 20.2}: Although \textit{no} mechanical intelligence can evidence that G\"{o}del's formula $[(\forall x)R(x)]$ is algorithmically \textit{verifiable}, a human intelligence \textit{can} evidence that $[(\forall x)R(x)]$ is algorithmically \textit{verifiable}; and
			
			\item[(ii)] \cite{An21}, \S 20.E, \textbf{Query 21} (Turing Test): Can you prove that, for any well-defined numeral $[n]$, G\"{o}del's arithmetic formula $[R(n)]$ is a theorem in the first-order Peano Arithmetic PA, where $[R(x)]$ is defined by its G\"{o}del number $r$ in eqn.12, and $[(\forall x)R(x)]$ is defined by its G\"{o}del number $17Gen\ r$ in eqn.13, on p.25 of \cite{Go31}? Answer only either `Yes' or `No'.
			
			\vspace{+1ex}
			\textit{Logician}: Yes.
			
			\begin{itemize}
				\item[ ] \textit{Reason}: By G\"{o}del's \textit{meta-mathematical} reasoning on p.26(2) of \cite{Go31}, a logician can conclude that, if a numeral $[n]$ is well-defined, then the formula $[R(n)]$ is a theorem in PA; even though the formula $[(\forall x)R(x)]$ is \textit{not} a theorem in PA.
			\end{itemize}
			
			\textit{Turing Machine}: No.
			
			\begin{itemize}
				\item[ ] \textit{Reason}: By \S \ref{sec:7}, Corollary \ref{sec:6.3.cor.1}, the formula $[\neg (\forall x)R(x)]$ is provable in PA and so, by the Provability Theorem for PA (\S \ref{sec:6.3.a}, Theorem \ref{sec:6.3.thm.1}), \textit{no} Turing machine can \textit{prove} that the formula $[(\forall x)R(x)]$ with G\"{o}del number $17Gen\ r$ is a theorem in PA and, ipso facto, conclude that, for any well-definable numeral $[n]$, G\"{o}del's arithmetic formula $[R(n)]$ is a theorem in PA.
			\end{itemize}			
		\end{enumerate}
	}.

	\section{Revisiting an \textit{evidence-based} paradigm}
	\label{sec:revisiting}
	
	We begin by revisiting the \textit{evidence-based} paradigm introduced in \cite{An12} and the paper \cite{An16}; a paradigm whose philosophical significance is that it pro-actively addresses the challenge\footnote{For a brief review of such challenges, see \cite{Fe06} and \cite{Fe08}; also \cite{An04} and \cite{Fre18}.}  which arises when an intelligence---whether human or mechanistic---accepts arithmetical propositions as \textit{true} under an interpretation---either axiomatically or on the basis of \textit{subjective} self-evidence---\textit{without} any specified methodology for \textit{objectively evidencing} such acceptance \textit{mechanically} in the sense of, for instance, Chetan Murthy and Martin L\"{o}b:
		
		\begin{quote}
			\footnotesize
			``It is by now folklore \ldots that one can view the \textit{values} of a simple functional language as specifying \textit{evidence} for propositions in a constructive logic \ldots" \\ \noindent \textit{\tiny{\ldots Murthy: \cite{Mu91}, \S 1 Introduction.}}
			
			\vspace{+1ex}
			\noindent ``Intuitively we require that for each event-describing sentence, $\phi_{o^{\iota}}n_{\iota}$ say (i.e. the concrete object denoted by $n_{\iota}$ exhibits the property expressed by  $\phi_{o^{\iota}}$), there shall be an algorithm (depending on \textbf{I}, i.e. $M^{*}$) to decide the truth or falsity of that sentence." \\ \noindent \textit{\tiny{\ldots L\"{o}b: \cite{Lob59}, p.165.}}
		\end{quote}
	
	\vspace{+1ex}
	The significance of the \textit{evidence-based} definitions of arithmetical truth, introduced in \cite{An12}, is that the first-order Peano Arithmetic PA---which, by \cite{An16}, Theorem 6.7, is \textit{finitarily} consistent\footnote{As sought (see \cite{Hi00}) by David Hilbert for the second of the twenty-three problems that he highlighted at the International Congress of Mathematicians in Paris in 1900; see also \S \ref{sec:pa.axms.true.alg.cmp}, Theorem \ref{sec:6.thm.2}.}---forms the bedrock for all formal mathematical languages that admit rational and real numbers\footnote{See, for instance, Edmund Landau's classically concise exposition \cite{La29} on the foundations of analysis; see also \cite{An21}, \S 22.C.c: \textit{A finitary perspective of the structure $\mathbb{N}$ of the natural numbers}.}.
	
		\begin{quote}
		\noindent \textbf{Axioms and rules of inference of the first-order Peano Arithmetic PA}\footnote{We use square brackets to differentiate between a symbolic expression $[F(x)]$---which denotes a formula of a formal language $L$---and the symbolic expression $F^{*}(x)$---which denotes its `semantics/meaning' under a \textit{well-defined} interpretation.}
		
		\begin{enumerate}
			\item[\textbf{PA$_{1}$}] $[(x_{1} = x_{2}) \rightarrow ((x_{1} = x_{3}) \rightarrow (x_{2} = x_{3}))]$;
			
			\item[\textbf{PA$_{2}$}] $[(x_{1} = x_{2}) \rightarrow (x_{1}^{\prime} = x_{2}^{\prime})]$;
			
			\item[\textbf{PA$_{3}$}] $[0 \neq x_{1}^{\prime}]$;
			
			\item[\textbf{PA$_{4}$}] $[(x_{1}^{\prime} = x_{2}^{\prime}) \rightarrow (x_{1} = x_{2})]$;
			
			\item[\textbf{PA$_{5}$}] $[( x_{1} + 0) = x_{1}]$;
			
			\item[\textbf{PA$_{6}$}] $[(x_{1} + x_{2}^{\prime}) = (x_{1} + x_{2})^{\prime}]$;
			
			\item[\textbf{PA$_{7}$}] $[( x_{1} \star 0) = 0]$;
			
			\item[\textbf{PA$_{8}$}] $[( x_{1} \star x_{2}^{\prime}) = ((x_{1} \star x_{2}) + x_{1})]$;
			
			\item[\textbf{PA$_{9}$}] For any well-formed formula $[F(x)]$ of PA: \\ $[F(0) \rightarrow (((\forall x)(F(x) \rightarrow F(x^{\prime}))) \rightarrow (\forall x)F(x))]$.
		\end{enumerate}
		
		\textbf{Generalisation in PA} If $[A]$ is PA-provable, then so is $[(\forall x)A]$.
		
		\vspace{+1ex}
		\textbf{Modus Ponens in PA} If $[A]$ and $[A \rightarrow B]$ are PA-provable, then so is $[B]$.		
		\end{quote}
		
	Hence the reliability of any conceptual metaphors\footnote{In the sense of \cite{LR00}; see also \cite{An21}, \S 25: \textit{The significance of \textit{evidence-based} reasoning for Cognitive Science}.} seeking to unambiguously express, and categorically communicate, our observations of physical phenomena which appeal---in their mathematical representations---to mathematical definitions of real numbers, must be circumscribed by whether, or not, PA can be interpreted \textit{categorically}, in some practicable sense\footnote{See \cite{An16}, \textbf{Corollary 7.2}: PA is categorical with respect to algorithmic \textit{computability} (see also \S \ref{sec:6.3.a}, Corollary \ref{cor:pa.categorical}).}, over the domain $\mathbb{N}$ of the natural numbers.
	
	\vspace{+1ex}
	Now conventional wisdom, whilst accepting Alfred Tarski's classical definitions\footnote{See \cite{Ta35}.} of the satisfiability and truth of the formulas of a formal language, under a \textit{well-defined} interpretation, as adequate to the intended purpose, \textit{postulates} that under the classical, putatively \textit{standard}, interpretation  $\mathcal{I}_{PA(\mathbb{N},\ S)}$ of the first-order Peano Arithmetic PA over the domain $\mathbb{N}$ of the natural numbers:
	
	\begin{enumerate}
		\item[(i)] The satisfiability/truth of the \textit{atomic} formulas of PA can be \textit{assumed} as \textit{uniquely} decidable under $\mathcal{I}_{PA(\mathbb{N},\ S)}$;
		
		\item[(ii)] The PA \textit{axioms} can be \textit{assumed} to \textit{uniquely} interpret as satisfied/true under $\mathcal{I}_{PA(\mathbb{N},\ S)}$;
		
		\item[(iii)] The PA \textit{rules of inference}---Generalisation and Modus Ponens---can be \textit{assumed} to \textit{uniquely} preserve such satisfaction/truth under $\mathcal{I}_{PA(\mathbb{N},\ S)}$;
		
		\item[(iv)] Aristotle's \textit{particularisation} can be \textit{assumed} to hold under $\mathcal{I}_{PA(\mathbb{N},\ S)}$.
		
		\begin{definition}[Aristotle's particularisation]
			\label{def:ap}
			If the formula $[\neg (\forall x) \neg F(x)]$\footnote{We note that, in a formal language, the formula `$[(\exists x)F(x)]$' is merely an abbreviation of the formula `$[\neg (\forall x)\neg F(x)]$'.} of a formal first order language L is defined as `true' under an interpretation, then we may always conclude \textit{unrestrictedly} that there must be some \textit{well-definable}, albeit \textit{unspecified}, object $s$ in the domain $D$ of the interpretation such that, if the formula $[F(x)]$ interprets as the relation $F^{*}(x)$ in $D$, then the proposition $F^{*}(s)$ is `true' under the interpretation.
		\end{definition}
		
		\begin{quote}
			\footnotesize			
			\textbf{Comment}. We note that Aristotle's particularisation entails\footnote{See \cite{An21}, \S 8.D., \textbf{Corollary 8.9}: If PA is consistent and Aristotle's particularisation holds over $\mathbb{N}$, then PA is $\omega$-consistent.} the \textit{non-finitary} assumption that if the classical first-order logic FOL\footnote{We take FOL to be a first-order predicate calculus such as the formal system $K$ defined in \cite{Me64}, p.57. For reasons detailed in \cite{An21}, \S 14.H.i (\textit{The misleading entailment of the fixed point theorem}) we prefer \cite{Me64} to \cite{Me15}.} is consistent, then it is $\omega$-consistent; thus we may always interpret the formal expression `$(\exists x)F(x)$' of a formal language under any \textit{well-defined} interpretation of FOL as `There exists an \textit{unspecified} object $s$ in the domain of the interpretation such that $F^{*}(s)$', where the formula $[F(x)]$ of the formal language interprets as $F^{*}(x)$.
			
			\vspace{+1ex}
			The far reaching consequences of the term \textit{unspecified}---in \S \ref{sec:intro.prov.thm.pvnp}, Definition \ref{def:ap}---for Rosser's `extension', in \cite{Ro36}, of G\"{o}del's proof of `undecidability' in \cite{Go31}, are highlighted by \S \ref{sec:7}, Corollary \ref{lem:ruleC.equiv.omega.cons}.
			
			\begin{quote}
			\textbf{$\omega$-consistent}. A formal system S is $\omega$-consistent if, and only if, there is no S-formula $[F(x)]$ for which, first, $[\neg (\forall x)F(x)]$ is S-provable and, second, $[F(a)]$ is S-provable for any \textit{specified} S-term $[a]$.			
			\end{quote}
		
			Moreover\footnote{See also \cite{An16}, p.36, fn.15.}, if Aristotle's particularisation holds in any finitary interpretation of PA over $\mathbb{N}$, then PA is $\omega$-consistent. Now, J. Barkley Rosser's `undecidable' arithmetical proposition in \cite{Ro36} is of the form $[(\forall y)(Q(h,y) \rightarrow (\exists z)(z \leq y \wedge S(h,z)))]$\footnote{See \cite{Ro36}, Theorem II, pp. 233–234; also \cite{Kl52}, Theorem 29, pp. 208–209, and \cite{Me64}, Proposition 3.32, pp. 145–146.}. Thus his `extension' of G\"{o}del's proof of undecidability does not yield a `formally undecidable proposition' in PA if PA is only \textit{simply} consistent; since the proof assumes that Aristotle's particularisation holds when interpreting the existential quantifier `$(\exists z)$' in $[(\forall y)(Q(h,y) \rightarrow (\exists z)(z \leq y \wedge S(h,z)))]$ under a finitary interpretation over $\mathbb{N}$, whence PA is implicitly assumed $\omega$-consistent.
		\end{quote}
	\end{enumerate}
	
	However, we shall see that the seemingly innocent and self-evident \textit{assumption} of \textit{uniqueness} in (i) to (iii) conceals an ambiguity with far-reaching consequences; as, equally if not more so, does the seemingly innocent \textit{assumption} in (iv) which, despite being obviously \textit{non-finitary}, is unquestioningly\footnote{See \cite{An21}, \S 7.B: \textit{Faith-based quantification}.} accepted in classical literature as equally self-evident under any logically unexceptionable interpretation of the classical first-order logic FOL.
	
	\vspace{+1ex}
	The ambiguity is revealed if we note that Tarski's classic definitions (see \S \ref{sec:rvwng.trski.ind.def} and \S \ref{sec:trski.ind.def}\footnote{See also \cite{An16}, \S 3: \textit{Reviewing Tarski's inductive assignment of truth-values under an interpretation}.}) permit both human and mechanistic intelligences to admit \textit{finitary}, i.e., \textit{evidence-based}, definitions of the satisfaction and truth of the \textit{atomic} formulas of PA over the domain $\mathbb{N}$ of the natural numbers in \textit{two} (as detailed in \S \ref{sec:alg.ver.int.pa.n} and \S \ref{sec:alg.cmp.int.pa}), hitherto unsuspected\footnote{However, we note that the \textit{non-finitary} (essentially \textit{faith-based} by \cite{An21}, \S 7.B), set-theoretically expressed, differentiation sought to be made between `intuitively learnable sets' and `intuitively computable sets' in \cite{CGK14}:
		
		\begin{quote}
			\footnotesize
			``We consider the notion of intuitive learnability and its relation to intuitive computability. We briefly discuss the Church's Thesis. We formulate the Learnability Thesis. Further we analyse the proof of the Church's Thesis presented by M. Mostowski. We indicate which assumptions of the Mostowski's argument implicitly include that the Church's Thesis holds. The impossibility of this kind of argument is strengthened by showing that the Learnability Thesis does not imply the Church's Thesis. Specifically, we show a \textit{natural} interpretation of intuitive computability under which intuitively learnable sets are exactly algorithmically learnable but intuitively computable sets form a proper superset of recursive sets." \\ \textit{\tiny{\ldots Czarnecki, Godziszewski, and Kaloci\'{n}ski: \cite{CGK14}, Abstract.}}			
		\end{quote}
	
	 \noindent appears, in intent, to mirror the \textit{finitary} (\textit{evidence-based} by \cite{An21}, \S 7.D), arithmetically-expressed, distinction between `algorithmically verifiable functions/relations' and `algorithmically computable functions/relations' introduced in \cite{An12}.} and essentially different, ways:
 	
	\begin{quote}
		(1a) In terms of \textit{classical} algorithmic \textit{verifiability} (see \S \ref{sec:intro.prov.thm.pvnp}, Definition \ref{sec:2.def.1}); and
		
		\begin{quote}
			\footnotesize
			\textbf{Comment}: `\textit{Classical}' since, as we argue in \S \ref{sec:ambgty.std.int.pa}, the classical, putatively \textit{standard}, interpretation $\mathcal{I}_{PA(\mathbb{N},\ S)}$ of PA (corresponding to Mendelson's `standard model for $S$' in \cite{Me64}, p.107) can be viewed as implicitly appealing to the algorithmic \textit{verifiability} of PA-formulas under the (\textit{well-defined}) interpretation $\mathcal{I}_{PA(\mathbb{N},\ SV)}$\footnote{As detailed in \cite{An16}, \S 5: \textit{The standard `verifiable' interpretation $\mathcal{I}_{PA(N,\ SV)}$ of PA over $N$} (see also \S \ref{sec:alg.ver.int.pa.n}).}.
		\end{quote}
		
		(1b) In terms of \textit{finitary} algorithmic \textit{computability} (see \S \ref{sec:intro.prov.thm.pvnp}, Definition \ref{sec:2.def.2});
	\end{quote}
	
	\noindent where we introduce the following \textit{evidence-based} (finitary) definitions:
	
	\begin{quote}
		\begin{definition}[Algorithmic verifiability]
			\label{sec:2.def.1}
			A number-theoretical relation $F(x)$ is algorithmically verifiable if, and only if, for any specifiable natural number $n$, there is a deterministic algorithm $AL_{(F,\ n)}$ which can provide objective evidence for deciding the truth/falsity of each proposition in the finite sequence $\{F(1), F(2), \ldots, F(n)\}$.
		\end{definition}
	
		\begin{definition}[Well-defined sequence]
			\label{def:const.wd.nt.seq}
			A Boolean number-theoretical sequence $\{F^{*}(1),$ $F^{*}(2),$ $\ldots\}$ is well-defined if, and only if, the number-theoretical relation $F^{*}(x)$ is algorithmically \textit{verifiable}.
		\end{definition}	
	
		\begin{definition}[Effective computability]
			\label{def:effective.computability}
			A number-theoretic function $F^{*}(x_{1}, \ldots, x_{n})$ is effectively computable if, and only if, $F^{*}(x_{1}, \ldots, x_{n})$ is well-defined.
		\end{definition}
		
		\begin{definition}[Integer specifiability]
			\label{sec:intro.def.1}
			An \textit{unspecified} natural number $n$ in $\mathbb{N}$ is \textit{specifiable} if, and only if, it can be explicitly denoted as a PA-numeral by a PA-formula that interprets as an algorithmically computable constant (natural number) in $\mathbb{N}$.
		\end{definition}
		
		\begin{quote}
			\footnotesize
			\textbf{Comment}: The significance of Definition \ref{sec:intro.def.1} is highlighted in \cite{An21}, Theorem 10.3\footnote{\cite{An21}, \S 10.C.a, \textbf{Theorem 10.3}: The Church-Turing Thesis entails Aristotle's particularisation.}.
		\end{quote}
		
		\begin{definition}[Deterministic algorithm]
			\label{def:det.alg}
			A deterministic algorithm is a well-defined mechanical method, such as a Turing machine, that computes a mathematical function which has a unique value for any input in its domain, and the algorithm is a process that produces this particular value as output.		
		\end{definition}
		
		\begin{quote}
			\footnotesize
			\textbf{Comment}: By `deterministic algorithm' we mean a `\textit{realizer}' in the sense of the \textit{Brouwer-Heyting-Kolmogorov} rules (see Stephen Cole Kleene's \cite{Kl52}, p.503-505).	
		\end{quote}
		
		\begin{definition}[Algorithmic computability]
			\label{sec:2.def.2}
			A number theoretical relation $F(x)$ is algorithmically computable if, and only if, there is a deterministic algorithm $AL_{F}$ that can provide objective evidence for deciding the truth/falsity of each proposition in the denumerable sequence $\{F(1), F(2), \ldots\}$.
		\end{definition}
		
		\begin{quote}
			\footnotesize
			\textbf{Comment}: In \cite{An21}, \S 7.G, Theorem 7.2\footnote{Corresponding to \cite{An16}, \textbf{Theorem 2.1}: There are number theoretic functions that are algorithmically \textit{verifiable} but not algorithmically \textit{computable}.}, we show that there are \textit{well-defined} number theoretic Boolean functions that are algorithmically \textit{verifiable} but not algorithmically \textit{computable}; and consider some consequences for the classical Church-Turing Thesis in \cite{An21}, \S 7.H.b.
		\end{quote}
	\end{quote}	

	We can now see the ambiguity in the unqualified \textit{assumption} of an \textit{unspecified} object $s$ in Aristotle's particularisation since, if the PA formula $[(\forall x)F(x)]$\footnote{For ease of exposition we consider, without loss of generality, only the case of a PA-formula with a single variable.} is intended\footnote{See also \cite{An21}, \S 10.C: \textit{Is the PA-formula $[(\forall x)F(x)]$ to be interpreted \textit{weakly} or \textit{strongly}?}.} to be read \textit{weakly} as \textit{`For any specified $x$, $F^{*}(x)$ is decidable as true'} under a \textit{well-defined} interpretation, where the formula $[F(x)]$ interprets as the arithmetical relation $F^{*}(x)$, then it must be consistently interpreted \textit{weakly} in terms of algorithmic \textit{verifiability} as follows:
	
	\begin{quote}
		\begin{definition}[Weak quantification]
			\label{def:ebr.weak.qntfcn}
			A PA formula $[(\forall x)F(x)]$ is algorithmically \textit{verifiable} as true under a \textit{well-defined} interpretation if, and only if, $F^{*}(x)$ is algorithmically \textit{verifiable} as always true.
		\end{definition}	
	\end{quote}
	
	Whereas, if the PA formula $[\neg (\forall x)F(x)]$ is intended to be read \textit{weakly} as \textit{`It is not the case that, for any specified $x$, $F^{*}(x)$ is true'}, then it must be consistently interpreted \textit{weakly} in terms of algorithmic \textit{verifiability} as:
	
	\begin{quote}
		\begin{definition}[Weak negation]
			\label{def:weak.neg}
			The PA formula $[\neg (\forall x)F(x)]$ is algorithmically \textit{verifiable} as true under a \textit{well-defined} interpretation if, and only if, there is no algorithm which will evidence that $[(\forall x)F(x)]$ is algorithmically \textit{verifiable} as always true under the interpretation.
		\end{definition}
	
		\vspace{+1ex}
		\footnotesize
		\textbf{Comment}: We note that \textit{weak negation} implies that $[(\forall x)F(x)]$ is not provable in PA; it \textit{does not}, however, entail that $F^{*}(x)$ is not algorithmically \textit{verifiable} as always true.
	\end{quote}

	Similarly, if $[(\forall x)F(x)]$ is intended to be read \textit{strongly} as \textit{`For all $x$, $F^{*}(x)$ is decidable as true'}, then it must be consistently interpreted \textit{strongly} in terms of algorithmic \textit{computability} as follows:
	
	\begin{quote}
		\begin{definition}[Strong quantification]
			\label{def:ebr.strong.qntfcn}
			A PA formula $[(\forall x)F(x)]$ is algorithmically \textit{computable} as true under a \textit{well-defined} interpretation if, and only if, $F^{*}(x)$ is algorithmically \textit{computable} as always true.
		\end{definition}	
	\end{quote}

	Whilst if the PA formula $[\neg (\forall x)F(x)]$ is intended to be read \textit{strongly} as \textit{`It is not the case that, for all $x$, $F^{*}(x)$ is true'}, then it must be consistently interpreted \textit{strongly} in terms of algorithmic \textit{computability} as:
	
	\begin{quote}
		\begin{definition}[Strong negation]
			\label{def:strong.neg}
			The PA formula $[\neg (\forall x)F(x)]$ is algorithmically \textit{computable} as true under a \textit{well-defined} interpretation if, and only if, there is no algorithm which will evidence that $[(\forall x)F(x)]$ is algorithmically \textit{computable} as always true under the interpretation.
		\end{definition}
	
		\vspace{+1ex}
		\footnotesize
		\textbf{Comment}: We note that \textit{strong negation}, too, implies that $[(\forall x)F(x)]$ is not provable in PA. By \S \ref{sec:6.3.a}, Theorem \ref{sec:6.3.thm.1} (\textit{Provability Theorem for PA}), it \textit{does}, however, entail that $F^{*}(x)$ is not algorithmically \textit{computable} as always true.
	\end{quote}
	
	We note that \textit{strong} algorithmic \textit{computability} implies the existence of an algorithm that can \textit{finitarily} decide the truth/falsity of each proposition in a \textit{well-defined} denumerable sequence of number-theoretical propositions; whereas \textit{weak} algorithmic \textit{verifiability} does not imply the existence of an algorithm that can \textit{finitarily} decide the truth/falsity of each proposition in a \textit{well-defined} denumerable sequence of number-theoretical propositions.

	\subsection{Reviewing Tarski's inductive assignment of truth-values under an interpretation}
	\label{sec:rvwng.trski.ind.def}
	
	The paper \cite{An16} essentially appeals to Definitions \ref{sec:2.def.1} to \ref{def:strong.neg} when addressing standard expositions\footnote{Such as in \cite{Me64}, pp.49-53.} of Tarski's inductive definitions\footnote{See \cite{Ta35}.} on the `satisfiability' and `truth' of the formulas of a formal language under an interpretation where:
	
	\begin{definition}
		\label{sec:5.def.1}
		If $[A]$ is an atomic formula $[A(x_{1}, x_{2}, \ldots, x_{n})]$ of a formal language S, then the denumerable sequence $(a_{1}, a_{2}, \ldots)$ in the domain $\mathbb{D}$ of an interpretation $\mathcal{I}_{S(\mathbb{D})}$ of S satisfies $[A]$ if, and only if:
		
		\begin{enumerate}
			\item[(i)] $[A(x_{1}, x_{2}, \ldots, x_{n})]$ interprets under $\mathcal{I}_{S(\mathbb{D})}$ as a unique relation $A^{*}(x_{1}, x_{2},$ $\ldots, x_{n})$ in $\mathbb{D}$ for any witness $\mathcal{W}_{\mathbb{D}}$ of $\mathbb{D}$;
			
			\item[(ii)] there is a Satisfaction Method that provides objective \textit{evidence}\footnote{In the sense of \cite{Mu91} and \cite{Lob59} (see \S \ref{sec:intro.prov.thm.pvnp}).} by which any witness $\mathcal{W}_{\mathbb{D}}$ of $\mathbb{D}$ can objectively \textbf{define} for any atomic formula $[A(x_{1}, x_{2}, \ldots, x_{n})]$ of S, and any given denumerable sequence $(b_{1}, b_{2}, \ldots)$ of $\mathbb{D}$, whether the proposition $A^{*}(b_{1}, b_{2}, \ldots, b_{n})$ holds or not in $\mathbb{D}$;
			
			\item[(iii)] $A^{*}(a_{1}, a_{2}, \ldots, a_{n})$ holds in $\mathbb{D}$ for any $\mathcal{W}_{\mathbb{D}}$.
		\end{enumerate}
	\end{definition}
	
	\begin {quote}
	\footnotesize
	\textbf{Witness:} From an \textit{evidence-based} perspective, the existence of a `witness' as in (i) above is implicit in the usual expositions of Tarski's definitions.
	
	\vspace{+1ex}
	\textbf{Satisfaction Method:} From an \textit{evidence-based} perspective, the existence of a Satisfaction Method as in (ii) above is also implicit in the usual expositions of Tarski's definitions.
	
	\vspace{+1ex}
	\textbf{An \textit{evidence-based} perspective:} We highlight the word `\textit{\textbf{define}}' in (ii) above to emphasise the \textit{evidence-based} perspective underlying this paper; which is that the concepts of `satisfaction' and `truth' under an interpretation are to be explicitly viewed as objective assignments by a convention that is witness-independent. A Platonist perspective would substitute `decide' for `define', thus implicitly suggesting that these concepts can `exist', in the sense of needing to be discovered by some witness-dependent means---eerily akin to a `revelation'---if the domain $\mathbb{D}$ is $\mathbb{N}$.
\end{quote}

\subsubsection{Tarski's inductive definitions}
\label{sec:trski.ind.def}

The truth values of `satisfaction', `truth', and `falsity' are then \textit{assumed} assignable inductively---whether \textit{finitarily} or \textit{non-finitarily}---to the compound formulas of  a first-order theory S under the interpretation $\mathcal{I}_{S(\mathbb{D})}$ in terms of \textit{only} the satisfiability of the atomic formulas of S over $\mathbb{D}$\footnote{See \cite{Me64}, p.51; \cite{Mu91}.}:

\begin{definition}
\label{sec:5.def.2}
A denumerable sequence $s$ of $\mathbb{D}$ satisfies $[\neg A]$ under $\mathcal{I}_{S(\mathbb{D})}$ if, and only if, $s$ does not satisfy $[A]$;
\end{definition}

\begin{definition}
\label{sec:5.def.3}
A denumerable sequence $s$ of $\mathbb{D}$ satisfies $[A \rightarrow B]$ under $\mathcal{I}_{S(\mathbb{D})}$ if, and only if, either it is not the case that $s$ satisfies $[A]$, or $s$ satisfies $[B]$;
\end{definition}

\begin{definition}
\label{sec:5.def.4}
A denumerable sequence $s$ of $\mathbb{D}$ satisfies $[(\forall x_{i})A]$ under $\mathcal{I}_{S(\mathbb{D})}$ if, and only if, given any denumerable sequence $t$ of $\mathbb{D}$ which differs from $s$ in at most the $i$'th component, $t$ satisfies $[A]$;
\end{definition}

\begin{definition}
\label{sec:5.def.5}
A well-formed formula $[A]$ of $\mathbb{D}$ is true under $\mathcal{I}_{S(\mathbb{D})}$ if, and only if, given any denumerable sequence $t$ of $\mathbb{D}$, $t$ satisfies $[A]$;
\end{definition}

\begin{definition}
\label{sec:5.def.6}
A well-formed formula $[A]$ of $\mathbb{D}$ is false under $\mathcal{I}_{S(\mathbb{D})}$ if, and only if, it is not the case that $[A]$ is true under $\mathcal{I}_{S(\mathbb{D})}$.
\end{definition}

The implicit, non-finitary, assumption of Aristotle's particularisation (\S \ref{sec:intro.prov.thm.pvnp}, Definition \ref{def:ap}) in current mathematical paradigms is evidenced in (V)(ii) of Mendelson's assertion---following his formulation of Tarski's definitions (essentially as above)---that:

\begin{quote}
\footnotesize
``Verification of the following consequences of the definitions above is left to the reader. (Most of the results are also obvious if one wishes to use only the ordinary intuitive understanding of the notions of truth and satisfaction). \ldots

\vspace{+1ex}
(V) (i) A sequence $s$ satisfies $\mathcal{A} \wedge \mathcal{B}$ if and only if $s$ satisfies $\mathcal{A}$ and $s$ satisfies $\mathcal{B}$. A sequence $s$ satisfies $\mathcal{A} \vee \mathcal{B}$ if and only if $s$ satisfies $\mathcal{A}$ or $s$ satisfies $\mathcal{B}$. A sequence $s$ satisfies $\mathcal{A} \equiv \mathcal{B}$ if and only if $s$ satisfies both $\mathcal{A}$ and $\mathcal{B}$ or $s$ satisfies neither $\mathcal{A}$ nor $\mathcal{B}$.

\vspace{+1ex}
(ii) A sequence $s$ satisfies $(Ex_{i})\mathcal{A}$ if and only if there is a sequence $s'$ which differs from $s$ in at most the $i^{th}$ place such that $s'$ satisfies $\mathcal{A}$." \\ \textit{\tiny{\ldots Mendelson: \cite{Me64}, pp.51-52.}}
\end{quote}

\subsubsection{The ambiguity in the classical \textit{standard} interpretation of PA over $\mathbb{N}$}
\label{sec:ambgty.std.int.pa}

Now, the classical \textit{standard} interpretation $\mathcal{I}_{PA(\mathbb{N},\ S)}$ of PA over the domain $\mathbb{N}$ of the natural numbers (corresponding to Mendelson's `standard model for $S$' in \cite{Me64}, p.107) is obtained if, in $\mathcal{I}_{S(\mathbb{D})}$:

\begin{enumerate}
\item[(a)] we define S as PA with the standard first-order predicate calculus FOL as the underlying logic\footnote{Where the string $[(\exists \ldots)]$ is defined as---and is to be treated as an abbreviation for---the PA formula $[\neg (\forall \ldots) \neg]$. We do not consider the case where the underlying logic is Hilbert's formalisation of Aristotle's logic of predicates in terms of his $\varepsilon$-operator (\cite{Hi27}, pp.465-466).};

\item[(b)] we define $\mathbb{D}$ as the set $\mathbb{N}$ of natural numbers;

\item[(c)] we assume for any \textit{atomic} formula $[A(x_{1}, x_{2}, \ldots, x_{n})]$ of PA, and any given sequence  $(b_{1}^{*}, b_{2}^{*}, \ldots, b_{n}^{*})$ of $\mathbb{N}$, that the proposition $A^{*}(b_{1}^{*}, b_{2}^{*}, \ldots, b_{n}^{*})$ is decidable in $\mathbb{N}$;

\item[(d)] we define the witness $\mathcal{W}_{(\mathbb{N},\ S)}$ \textit{informally} as the `mathematical intuition' of a human intelligence for whom, classically, (c) has been \textit{implicitly} accepted as `\textit{objectively decidable}' in $\mathbb{N}$.

\item[(e)] we postulate that Aristotle's particularisation\footnote{Aristotle's particularisation (\S \ref{sec:intro.prov.thm.pvnp}, Definition \ref{def:ap}) entails that a PA formula such as $[(\exists x)F(x)]$ can always be taken to interpret under $\mathcal{I}_{PA(\mathbb{N},\ S)}$ as `There is some natural number $n$ such that $F(n)$ holds in $\mathbb{N}$'.} holds over $\mathbb{N}$.

\vspace{+1ex}
\footnotesize{
	\textbf{Comment}: Clearly, (e) (which, in \cite{Me64}, is implicitly entailed by \cite{Me64}, para (V)(ii), p.52) does not form any part of Tarski's inductive definitions of the satisfaction, and truth, of the formulas of PA under the above interpretation. Moreover, its inclusion makes $\mathcal{I}_{PA(\mathbb{N},\ S)}$ extraneously non-finitary\footnote{As argued by Brouwer in \cite{Br08}; see also \cite{An21}, \S 7.B: \textit{\textit{Faith-based} quantification}.}.
}
\end{enumerate}

We shall show that the implicit acceptance in (d) conceals an ambiguity that needs to be made explicit since:

\begin{lemma}
\label{sec:4.lem.1}
Under the interpretation $\mathcal{I}_{PA(\mathbb{N},\ S)}$, an atomic formula $A^{*}(x_{1}, x_{2}, \ldots, x_{n})$ would be \textit{evidenced} as both algorithmically verifiable and algorithmically computable in $\mathbb{N}$ by any witness $\mathcal{W}_{(\mathbb{N},\ S)}$.
\end{lemma}

\begin{proof} (i) It follows from the argument in \S \ref{sec:alg.ver.int.pa.n}, Theorem \ref{sec:5.4.lem.1}, that $A^{*}(x_{1}, x_{2}, \ldots, x_{n})$ would be \textit{evidenced} as algorithmically \textit{verifiable} in $\mathbb{N}$ by any witness $\mathcal{W}_{(\mathbb{N},\ S)}$.

\vspace{+1ex}
\noindent (ii) It follows from the argument in \S \ref{sec:alg.cmp.int.pa}, Theorem \ref{sec:5.5.lem.2}, that $A^{*}(x_{1}, x_{2},$ $\ldots, x_{n})$ would be \textit{evidenced} as algorithmically \textit{computable} in $\mathbb{N}$ by any witness $\mathcal{W}_{(\mathbb{N},\ S)}$. The lemma follows.\hfill $\Box$
\end{proof}

\subsection{The \textit{weak}, algorithmically \textit{verifiable}, standard interpretation $\mathcal{I}_{PA(\mathbb{N},\ SV)}$ of PA}
\label{sec:alg.ver.int.pa.n}

We note that conventional wisdom \textit{implicitly} considers, as \textit{the} classical \textit{standard} interpretation $\mathcal{I}_{PA(\mathbb{N},\ S)}$ of PA (corresponding to Mendelson's `standard model for $S$' in \cite{Me64}, p.107), a \textit{weak}, algorithmically \textit{verifiable}, interpretation $\mathcal{I}_{PA(\mathbb{N},\ SV)}$ of PA\footnote{See \cite{An16}, \S 5: \textit{The standard `verifiable' interpretation $\mathcal{I}_{PA(N,\ SV)}$ of PA over $N$}.}, where:

\begin{definition}
\label{sec:5. def.1}
An atomic formula $[A]$ of PA is satisfiable under the interpretation $\mathcal{I}_{PA(\mathbb{N},\ SV)}$ if, and only if, $[A]$ is algorithmically verifiable under $\mathcal{I}_{PA(\mathbb{N},\ SV)}$.
\end{definition}

\noindent It follows that\footnote{See \cite{An16}, \textbf{Theorem 5.1}: \textit{The atomic formulas of PA are algorithmically \textit{verifiable} as true or false under the standard interpretation $\mathcal{I}_{PA(\mathbb{N},\ SV)}$}.}:

\begin{theorem}
\label{sec:5.5.lem.1}
The atomic formulas of PA are algorithmically verifiable as true or false under the algorithmically \textit{verifiable} interpretation $\mathcal{I}_{PA(\mathbb{N},\ SV)}$. 
\end{theorem}

\begin{proof} 
It follows by G\"{o}del's definition of the primitive recursive relation $xBy$\footnote{\cite{Go31}, p. 22(45).}---where $x$ is the G\"{o}del number of a proof sequence in PA whose last term is the PA formula with G\"{o}del-number $y$---that, if $[A]$ is an atomic formula of PA, we can algorithmically verify which one of the PA formulas $[A]$ and $[\neg A]$ is necessarily PA-provable and, ipso facto, true under $\mathcal{I}_{PA(\mathbb{N},\ SV)}$.\hfill $\Box$
\end{proof}

\begin{quote}
\footnotesize
\textbf{Comment}: We note that the interpretation $\mathcal{I}_{PA(\mathbb{N},\ SV)}$ cannot claim to be finitary\footnote{Since it defines a model of PA if, and only if, PA is $\omega$-consistent and so we may always non-finitarily conclude from $[(\exists x)R(x)]$ the existence of some numeral $[n]$ such that $[R(n)]$.}, since it follows from \cite{An16}, Theorem 2.1\footnote{\cite{An16}, \textbf{Theorem 2.1}: There are number theoretic functions that are algorithmically verifiable but not algorithmically computable. See also \cite{An21}, \S 7.G, Theorem 7.2.}, that we cannot conclude finitarily from Tarski's Definition \ref{sec:5.def.1} (in \S \ref{sec:rvwng.trski.ind.def}), and Definitions \ref{sec:5.def.2} to \ref{sec:5.def.6} (in \S \ref{sec:trski.ind.def}), whether or not a quantified PA formula $[(\forall x_{i})R]$ is algorithmically \textit{verifiable} as true under $\mathcal{I}_{PA(\mathbb{N},\ SV)}$ if $[R]$ is algorithmically \textit{verifiable}, but not algorithmically \textit{computable}, under $\mathcal{I}_{PA(\mathbb{N},\ SV)}$\footnote{Although a proof that such a PA formula exists is not obvious, by \cite{An16}, Corollary 8.3, p.42 (see also \S \ref{sec:7}, Corollary \ref{sec:6.3.cor.1.1}), G\"{o}del's `undecidable' arithmetical formula $[R(x)]$ is algorithmically \textit{verifiable}, but not algorithmically \textit{computable}, under the interpretation $\mathcal{I}_{PA(\mathbb{N},\ SV)}$.}.
\end{quote}

\subsubsection{The PA axioms are algorithmically \textit{verifiable} as true under $\mathcal{I}_{PA(\mathbb{N},\ SV)}$}
\label{sec:pa.axms.true.alg.ver}

Defining satisfaction in terms of algorithmic \textit{verifiability} under $\mathcal{I}_{PA(\mathbb{N},\ SV)}$ then entails\footnote{See \cite{An16}, \S 5: \textit{The standard `verifiable' interpretation $\mathcal{I}_{PA(N,\ SV)}$ of PA over $N$}.}:

\begin{lemma}
\label{sec:5.4.lem.1}
The PA axioms PA$_{1}$ to PA$_{8}$ are algorithmically verifiable as true over $\mathbb{N}$ under the interpretation $\mathcal{I}_{PA(\mathbb{N},\ SV)}$.
\end{lemma}

\begin{proof}
Since $[x+y]$, $[x \star y]$, $[x = y]$, $[{x^{\prime}}]$ are defined recursively\footnote{cf. \cite{Go31}, p.17.}, the PA axioms PA$_{1}$ to PA$_{8}$ (see \S \ref{sec:intro.prov.thm.pvnp}) interpret as recursive relations that do not involve any quantification. The lemma follows from \S \ref{sec:alg.ver.int.pa.n}, Theorem \ref{sec:5.5.lem.1}, Tarski's Definition \ref{sec:5.def.1} (in \S \ref{sec:rvwng.trski.ind.def}), and Tarski's Definitions \ref{sec:5.def.2} to \ref{sec:5.def.6} (in \S \ref{sec:trski.ind.def}).\hfill $\Box$
\end{proof}

\begin{lemma}
\label{sec:5.4.lem.2}
For any given PA formula $[F(x)]$, the Induction axiom schema $[F(0)$ $\rightarrow (((\forall x)(F(x)$ $\rightarrow F(x^{\prime}))) \rightarrow (\forall x)F(x))]$ interprets as an algorithmically verifiable true formula under $\mathcal{I}_{PA(\mathbb{N},\ SV)}$.
\end{lemma}

\begin{proof} We note that, by Tarski's Definition \ref{sec:5.def.1} (in \S \ref{sec:rvwng.trski.ind.def}), and Definitions \ref{sec:5.def.2} to \ref{sec:5.def.6} (in \S \ref{sec:trski.ind.def}):

\begin{enumerate}	
	\item[(a)] If $[F(0)]$ interprets as an algorithmically \textit{verifiable} false formula under $\mathcal{I}_{PA(\mathbb{N},\ SV)}$, the lemma is proved.
	
	\begin{quote}
		\footnotesize \textit{Reason}: Since $[F(0) \rightarrow (((\forall x)(F(x) \rightarrow F(x^{\prime}))) \rightarrow (\forall x)F(x))]$ interprets as an algorithmically \textit{verifiable} true formula under $\mathcal{I}_{PA(\mathbb{N},\ SV)}$ if, and only if, either $[F(0)]$ interprets as an algorithmically \textit{verifiable} false formula, or $[((\forall x)(F(x) \rightarrow F(x^{\prime}))) \rightarrow (\forall x)F(x)]$ interprets as an algorithmically \textit{verifiable} true formula, under under $\mathcal{I}_{PA(\mathbb{N},\ SV)}$.
	\end{quote}
	
	\item[(b)] If $[F(0)]$ interprets as an algorithmically \textit{verifiable} true formula, and $[(\forall x)(F(x) \rightarrow F(x^{\prime}))]$ interprets as an algorithmically \textit{verifiable} false formula, under $\mathcal{I}_{PA(\mathbb{N},\ SV)}$, the lemma is proved.
	
	\item[(c)] If $[F(0)]$ and $[(\forall x)(F(x) \rightarrow F(x^{\prime}))]$ both interpret as algorithmically \textit{verifiable} true formulas under $\mathcal{I}_{PA(\mathbb{N},\ SV)}$ then, for any \textit{specified} natural number $n$, there is an algorithm which (by Definition \ref{sec:2.def.1}) will evidence that $[F(n) \rightarrow F(n^{\prime})]$ is an algorithmically \textit{verifiable} true formula under $\mathcal{I}_{PA(\mathbb{N},\ SV)}$.
	
	\item[(d)] Since $[F(0)]$ interprets as an algorithmically \textit{verifiable} true formula under $\mathcal{I}_{PA(\mathbb{N},\ SV)}$, it follows, for any \textit{specified} natural number $n$, that there is an algorithm which will evidence that each of the formulas in the finite sequence $\{[F(0),\ F(1),\ \ldots,\ F(n)\}]$ is an algorithmically \textit{verifiable} true formula under the interpretation.
	
	\item[(e)] Hence $[(\forall x)F(x)]$ is an algorithmically \textit{verifiable} true formula under $\mathcal{I}_{PA(\mathbb{N},\ SV)}$.
\end{enumerate}

\noindent Since the above cases are exhaustive, the lemma follows.\hfill $\Box$
\end{proof}

\begin{quote}
\footnotesize
\textbf{Comment}: We note that if $[F(0)]$ and $[(\forall x)(F(x) \rightarrow F(x^{\prime}))]$ both interpret as algorithmically verifiable true formulas under $\mathcal{I}_{PA(\mathbb{N},\ SV)}$, then we can only conclude that, for any \textit{specified} natural number $n$, there is an algorithm, say TM$_{(F, n)}$, which will give evidence for any $m \leq n$ that the formula $[F(m)]$ is true under $\mathcal{I}_{PA(\mathbb{N},\ S)}$.

\vspace{+.5ex}
We cannot conclude that there is an algorithm TM$_{F}$ which, for any \textit{specified} natural number $n$, will give evidence that the formula $[F(n)]$ is true under $\mathcal{I}_{PA(\mathbb{N},\ S)}$.
\end{quote}

\begin{lemma}
\label{sec:5.4.lem.3}
Generalisation preserves algorithmically verifiable truth under $\mathcal{I}_{PA(\mathbb{N},\ SV)}$.
\end{lemma}

\begin{proof}
The two meta-assertions:

\vspace{+1ex}
`$[F(x)]$ interprets as an algorithmically \textit{verifiable} true formula under $\mathcal{I}_{PA(\mathbb{N},\ SV)}$\footnote{See Definition \ref{sec:5.def.5} (in \S \ref{sec:trski.ind.def})}'

\vspace{+1ex}
\noindent and

\vspace{+1ex}
`$[(\forall x)F(x)]$ interprets as an algorithmically \textit{verifiable} true formula under $\mathcal{I}_{PA(\mathbb{N},\ SV)}$'

\vspace{+1ex}
\noindent both mean:

\vspace{+1ex}
$[F(x)]$ is algorithmically verifiable as always true under $\mathcal{I}_{PA(\mathbb{N},\ SV)}$. \hfill $\Box$
\end{proof}

\noindent It is also straightforward to see that:

\begin{lemma}
\label{sec:5.4.lem.4}
Modus Ponens preserves algorithmically verifiable truth under $\mathcal{I}_{PA(\mathbb{N},\ SV)}$. \hfill $\Box$
\end{lemma}

\noindent We thus have that:

\begin{theorem}
\label{sec:5.4.lem.5}
The axioms of PA are always algorithmically verifiable as true under the interpretation $\mathcal{I}_{PA(\mathbb{N},\ SV)}$, and the rules of inference of PA preserve the properties of algorithmically verifiable satisfaction/truth under $\mathcal{I}_{PA(\mathbb{N},\ SV)}$.\hfill $\Box$
\end{theorem}

\noindent By \S \ref{sec:alg.ver.int.pa.n}, Theorem \ref{sec:5.5.lem.1} we further conclude that PA is \textit{weakly} consistent:

\begin{theorem}
\label{sec:5.4.thm.2}
If the PA formulas are algorithmically verifiable as true or false under $\mathcal{I}_{PA(\mathbb{N},\ SV)}$, then PA is consistent. \hfill $\Box$
\end{theorem}

\begin{quote}
\footnotesize
\textbf{Comment}: We note that, unlike Gentzen's argument\footnote{See \cite{Me64}, pp.258-271.}, which appeals to debatably `constructive' properties of set-theoretically defined transfinite ordinals, such a---strictly arithmetical---\textit{weak} proof of consistency is unarguably `constructive'; however it is not `finitary' since we cannot conclude from \S \ref{sec:alg.ver.int.pa.n}, Theorem \ref{sec:5.5.lem.1} that the quantified formulas of PA are `finitarily' decidable as true or false under the interpretation $\mathcal{I}_{PA(\mathbb{N},\ SV)}$.
\end{quote}

\subsection{The \textit{strong}, algorithmically \textit{computable}, interpretation $\mathcal{I}_{PA(\mathbb{N},\ SC)}$ of PA}
\label{sec:alg.cmp.int.pa}

We consider next\footnote{As detailed in \cite{An16}, \S 6: \textit{The standard `computable' interpretation $\mathcal{I}_{PA(N,\ SC)}$ of PA over $N$}.} a \textit{strong}, algorithmically \textit{computable}, interpretation $\mathcal{I}_{PA(\mathbb{N},\ SC)}$ of PA, under which we define:

\begin{definition}
\label{sec:5. def.2}
An atomic formula $[A]$ of PA is satisfiable under the interpretation $\mathcal{I}_{PA(\mathbb{N},\ SC)}$ if, and only if, $[A]$ is algorithmically computable under $\mathcal{I}_{PA(\mathbb{N},\ SC)}$.
\end{definition}

\noindent It follows that\footnote{See \cite{An16}, \textbf{Theorem 6.1}: The \textit{atomic} formulas of PA are algorithmically \textit{computable} as true or as false under the standard interpretation $\mathcal{I}_{PA(\mathbb{N},\ SC)}$.}:

\begin{theorem}
\label{sec:5.5.lem.2}
The atomic formulas of PA are algorithmically computable as true or as false under the algorithmically computable interpretation $\mathcal{I}_{PA(\mathbb{N},\ SC)}$. 
\end{theorem}

\begin{proof}
If $[A(x_{1}, x_{2}, \ldots, x_{n})]$ is an \textit{atomic} formula of PA then, for any given sequence of numerals $[b_{1}, b_{2}, \ldots, b_{n}]$, the PA formula $[A(b_{1}, b_{2},$ $\ldots, b_{n})]$ is an \textit{atomic} formula of the form $[c=d]$, where $[c]$ and $[d]$ are \textit{atomic} PA formulas that denote PA numerals. Since $[c]$ and $[d]$ are recursively defined formulas in the language of PA, it follows from a standard result\footnote{For any natural numbers $m,\ n$, if $m \neq n$, then PA proves $[\neg(m = n)]$ (\cite{Me64}, p.110, Proposition 3.6). The converse is obviously true.} that, by \S \ref{sec:intro.prov.thm.pvnp}, Definition \ref{sec:2.def.2}, $[c=d]$ is algorithmically \textit{computable} as either true or false in $\mathbb{N}$ since there is an algorithm that, for any given sequence of numerals $[b_{1}, b_{2}, \ldots, b_{n}]$, will give \textit{evidence}\footnote{In the sense of \cite{Mu91} and \cite{Lob59} (op. cit.).} whether $[A(b_{1}, b_{2},$ $\ldots, b_{n})]$ interprets as true or false in $\mathbb{N}$. The lemma follows.\hfill $\Box$
\end{proof}

\noindent We note that the interpretation $\mathcal{I}_{PA(\mathbb{N},\ SC)}$ is finitary since\footnote{See \cite{An16}, \textbf{Lemma 6.2}: The formulas of PA are algorithmically \textit{computable} as true or as false under $\mathcal{I}_{PA(\mathbb{N},\ SC)}$.}:

\begin{lemma}
\label{sec:5.5.lem.3}
The formulas of PA are algorithmically computable finitarily as true or as false under $\mathcal{I}_{PA(\mathbb{N},\ SC)}$.
\end{lemma}

\begin{proof}
The Lemma follows by finite induction from by \S \ref{sec:intro.prov.thm.pvnp}, Definition \ref{sec:2.def.2}, Tarski's Definition \ref{sec:5.def.1} (in \S \ref{sec:rvwng.trski.ind.def}), and Definitions \ref{sec:5.def.2} to \ref{sec:5.def.6} (in \S \ref{sec:trski.ind.def}), and Theorem \ref{sec:5.5.lem.2}.\hfill $\Box$
\end{proof}

\subsubsection{The PA axioms are algorithmically \textit{computable} as true under $\mathcal{I}_{PA(\mathbb{N},\ SC)}$}
\label{sec:pa.axms.true.alg.cmp}

Defining satisfaction in terms of algorithmic \textit{computability} under $\mathcal{I}_{PA(\mathbb{N},\ SC)}$ now entails\footnote{See \cite{An16}, \S 6: \textit{The standard `computable' interpretation $\mathcal{I}_{PA(N,\ SC)}$ of PA over $N$}.}:

\begin{lemma}
\label{sec:6.lem.1}
The PA axioms PA$_{1}$ to PA$_{8}$ are algorithmically computable as true under the interpretation $\mathcal{I}_{PA(\mathbb{N},\ SC)}$.
\end{lemma}

\begin{proof}
Since $[x+y]$, $[x \star y]$, $[x = y]$, $[{x^{\prime}}]$ are defined recursively\footnote{cf. \cite{Go31}, p.17.}, the PA axioms PA$_{1}$ to PA$_{8}$ (see \S \ref{sec:intro.prov.thm.pvnp}) interpret as recursive relations that do not involve any quantification. The lemma follows from \S \ref{sec:alg.ver.int.pa.n}, Theorem \ref{sec:5.5.lem.1} and Tarski's Definition \ref{sec:5.def.1} (in \S \ref{sec:rvwng.trski.ind.def}), and Definitions \ref{sec:5.def.2} to \ref{sec:5.def.6} (in \S \ref{sec:trski.ind.def}).\hfill $\Box$
\end{proof}

\begin{lemma}
\label{sec:6.lem.2}
For any given PA formula $[F(x)]$, the Induction axiom schema $[F(0)$ $\rightarrow (((\forall x)(F(x) \rightarrow F(x^{\prime}))) \rightarrow (\forall x)F(x))]$ interprets as an algorithmically computable true formula under $\mathcal{I}_{PA(\mathbb{N},\ SC)}$.
\end{lemma}

\begin{proof}
By Tarski's Definition \ref{sec:5.def.1} (in \S \ref{sec:rvwng.trski.ind.def}), and Definitions \ref{sec:5.def.2} to \ref{sec:5.def.6} (in \S \ref{sec:trski.ind.def}):
\begin{enumerate}
	
	\item[(a)] If $[F(0)]$ interprets as an algorithmically \textit{computable} false formula under $\mathcal{I}_{PA(\mathbb{N},\ SC)}$ the lemma is proved.
	
	\begin{quote}
		\footnotesize
		\textit{Reason}: Since $[F(0) \rightarrow (((\forall x)(F(x) \rightarrow F(x^{\prime}))) \rightarrow (\forall x)F(x))]$ interprets as an algorithmically \textit{computable} true formula if, and only if, either $[F(0)]$ interprets as an algorithmically \textit{computable} false formula, or $[((\forall x)(F(x) \rightarrow F(x^{\prime}))) \rightarrow (\forall x)F(x)]$ interprets as an algorithmically \textit{computable} true formula, under $\mathcal{I}_{PA(\mathbb{N},\ SC)}$.
	\end{quote}
	
	\item[(b)] If $[F(0)]$ interprets as an algorithmically \textit{computable} true formula, and $[(\forall x)(F(x) \rightarrow F(x^{\prime}))]$ interprets as an algorithmically \textit{computable} false formula, under $\mathcal{I}_{PA(\mathbb{N},\ SC)}$, the lemma is proved.
	
	\item[(c)] If $[F(0)]$ and $[(\forall x)(F(x) \rightarrow F(x^{\prime}))]$ both interpret as algorithmically \textit{computable} true formulas under $\mathcal{I}_{PA(\mathbb{N},\ SC)}$, then by Definition \ref{sec:2.def.2} there is an algorithm which, for any natural number $n$, will give \textit{evidence}\footnote{In the sense of \cite{Mu91} and \cite{Lob59} (op. cit.).} that the formula $[F(n) \rightarrow F(n^{\prime})]$ is an algorithmically \textit{computable} true formula under $\mathcal{I}_{PA(\mathbb{N},\ SC)}$.
	
	\item[(d)] Since $[F(0)]$ interprets as an algorithmically \textit{computable} true formula under $\mathcal{I}_{PA(\mathbb{N},\ SC)}$, it follows that there is an algorithm which, for any natural number $n$, will give \textit{evidence} that $[F(n)]$ is an algorithmically \textit{computable} true formula under the interpretation.
	
	\item[(e)] Hence $[(\forall x)F(x)]$ is an algorithmically \textit{computable} true formula under $\mathcal{I}_{PA(\mathbb{N},\ SC)}$.
\end{enumerate}

\noindent Since the above cases are exhaustive, the lemma follows.\hfill $\Box$
\end{proof}

\begin{lemma}
\label{sec:6.lem.3}
Generalisation preserves algorithmically computable truth under $\mathcal{I}_{PA(\mathbb{N},\ SC)}$.
\end{lemma}

\begin{proof}
The two meta-assertions:

\vspace{+1ex}
`$[F(x)]$ interprets as an algorithmically computable true formula under $\mathcal{I}_{PA(\mathbb{N},\ SC)}$\footnote{See \S 2.A.a, Definition \ref{sec:5.def.5}}'

\vspace{+1ex}
\noindent and

\vspace{+1ex}
`$[(\forall x)F(x)]$ interprets as an algorithmically computable true formula under $\mathcal{I}_{PA(\mathbb{N},\ SC)}$'

\vspace{+1ex}
\noindent both mean:

\vspace{+1ex}
$[F(x)]$ is algorithmically computable as always true under $\mathcal{I}_{PA(\mathbb{N},\ S)}$. \hfill $\Box$
\end{proof}

\noindent It is also straightforward to see that:

\begin{lemma}
\label{sec:6.lem.4}
Modus Ponens preserves algorithmically computable truth under $\mathcal{I}_{PA(\mathbb{N},\ SC)}$. \hfill $\Box$
\end{lemma}

\noindent We conclude that\footnote{Without appeal, moreover, to Aristotle's particularisation.}:

\begin{theorem}
\label{sec:6.lem.5}
The axioms of PA are always algorithmically computable as true under the interpretation $\mathcal{I}_{PA(\mathbb{N},\ SC)}$, and the rules of inference of PA preserve the properties of algorithmically computable satisfaction/truth under $\mathcal{I}_{PA(\mathbb{N},\ SC)}$.\hfill $\Box$
\end{theorem}

\noindent We thus have a \textit{finitary} proof that:

\begin{theorem}
\label{sec:6.thm.2}
PA is \textit{strongly} consistent. \hfill $\Box$
\end{theorem}

\begin{quote}
	\footnotesize
	\textbf{Poincar\'{e}-Hilbert debate}: We note that the Poincar\'{e}-Hilbert debate\footnote{See \cite{Hi27}, p.472; also \cite{Br13}, p.59; \cite{We27}, p.482; \cite{Pa71}, p.502-503.} on whether the PA Axiom Schema of Induction can be labelled `finitary' or not dissolves since:
	
	\begin{enumerate}
		\item[(i)] The algorithmically \textit{verifiable}, non-finitary, interpretation $\mathcal{I}_{PA(\mathbb{N},\ SV)}$ of PA validates Poincar\'{e}'s contention that the PA Axiom Schema of Finite Induction could not be justified finitarily with respect to algorithmic \textit{verifiability} under the classical \textit{standard} interpretation of arithmetic\footnote{See \cite{Me64}, p.107.}, as any such argument would necessarily need to appeal to some form of infinite induction\footnote{Such as, for instance, in Gerhard Gentzen's \textit{non-finitary} proof of consistency for PA, which involves a \textit{non-finitary} Rule of Infinite Induction (see \cite{Me64}, p.259(II)(e).) that appeals to the well-ordering property of transfinite ordinals.}; whilst
		
		\item[(ii)] The algorithmically \textit{computable} finitary interpretation $\mathcal{I}_{PA(\mathbb{N},\ SC)}$ of PA validates Hilbert's belief that a finitary justification of the Axiom Schema was possible under some finitary interpretation of an arithmetic such as PA.
	\end{enumerate}
	
\end{quote}

\subsection{Bridging PA Provability and Turing-computability}
\label{sec:6.3.a}

We can now show how \textit{evidence-based} reasoning admits a Provability Theorem for PA\footnote{See \cite{An16}, \textbf{Theorem 7.1} (Provability Theorem for PA): A PA formula $[F(x)]$ is PA-provable if, and only if, $[F(x)]$ is algorithmically \textit{computable} as always true in $\mathbb{N}$.} which allows us to bridge arithmetic provability, and Turing-computability; in the sense conjectured, for instance, by Christian S. Calude, Elena Calude and Solomon Marcus in \cite{CCS01}:

\begin{quote}
\footnotesize
``Classically, there are two equivalent ways to look at the mathematical notion of proof: logical, as a finite sequence of sentences strictly obeying some axioms and inference rules, and computational, as a specific type of computation. Indeed, from a proof given as a sequence of sentences one can easily construct a Turing machine producing that sequence as the result of some finite computation and, conversely, given a machine computing a proof we can just print all sentences produced during the computation and arrange them into a sequence." \\ \textit{\tiny{\ldots Calude, Calude and Marcus: \cite{CCS01}.}}
\end{quote}

\noindent where the authors seem to hold that Turing-computability of a `proof', in the case of a mathematical proposition, ought to be treated as equivalent to the provability of its representation in the corresponding formal language\footnote{We note that Theorem \ref{sec:6.3.thm.1} (Provability Theorem for PA) also offers a solution to Barendregt and Wiedijk's: `\textit{The challenge of computer mathematics}' \cite{BW05}.}.

\begin{theorem}[Provability Theorem for PA]
	\label{sec:6.3.thm.1}
	A PA formula $[F(x)]$ is PA-provable if, and only if, $[F(x)]$ is algorithmically \textit{computable} as always true in $\mathbb{N}$.
\end{theorem}

\begin{proof}
We have by definition that $[(\forall x)F(x)]$ interprets as true under the interpretation $\mathcal{I}_{PA(\mathbb{N},\ SC)}$ if, and only if, $[F(x)]$ is algorithmically \textit{computable} as always true in $\mathbb{N}$.

\vspace{+1ex}
\noindent By \S \ref{sec:pa.axms.true.alg.cmp}, Theorem  \ref{sec:6.lem.5}, $\mathcal{I}_{PA(\mathbb{N},\ SC)}$ defines a finitary model of PA over $\mathbb{N}$ such that:

\begin{enumerate}
	\item[(a) ] If $[(\forall x)F(x)]$ is PA-provable, then $[F(x)]$ interprets as an arithmetic relation $F^{*}(x)$ which is algorithmically \textit{computable} as always true in $\mathbb{N}$;
	
	\item[(b) ] If $[\neg(\forall x)F(x)]$ is PA-provable, then it is not the case that $[F(x)]$ interprets as an arithmetic relation $F^{*}(x)$ which  is algorithmically \textit{computable} as always true in $\mathbb{N}$.
\end{enumerate}

Now, we cannot have that both $[(\forall x)F(x)]$ and $[\neg(\forall x)F(x)]$ are PA-unprovable for some PA formula $[F(x)]$, as this would yield the contradiction:

\begin{enumerate}
	\item[(i)] There is a finitary model---say $\mathcal{I}'_{PA(\mathbb{N},\ SC)}$---of PA+$[(\forall x)F(x)]$ in which $[F(x)]$ interprets as an arithmetic relation $F^{*}(x)$ that is algorithmically \textit{computable} as always true in $\mathbb{N}$.
	
	\item[(ii)] There is a finitary model---say $\mathcal{I}''_{PA(\mathbb{N},\ SC)}$---of PA+$[\neg(\forall x)F(x)]$ in which $[F(x)]$ interprets as an arithmetic relation $F^{*}(x)$ that is \textit{not} algorithmically \textit{computable} as always true in $\mathbb{N}$.
\end{enumerate}

\noindent The theorem follows.\hfill $\Box$
\end{proof}

We note that, contradicting current paradigms, PA has no \textit{non-standard} models\footnote{For reasons detailed in \cite{An21}, \S 17: \textit{The significance of \textit{evidence-based} reasoning for non-standard models of PA}.}, as the Provability Theorem for PA immediately entails\footnote{See also \cite{An16}, \textbf{Corollary 7.2}: PA is categorical with respect to algorithmic \textit{computability}.}:

\begin{corollary}
\label{cor:pa.categorical}
	PA is categorical. \hfill $\Box$
\end{corollary}

\begin{quote}
	\footnotesize
	\textbf{Comment}: We further note, without proof\footnote{See, however, the three \textit{independent} proofs in \cite{An21}:
		
		\begin{enumerate}
			\item[(i)] \S 14.C, \textbf{Lemma 14.1}: The structure of the finite ordinals under any putative well-defined interpretation of ZF is not isomorphic to the structure $\mathbb{N}$ of the natural numbers;
			
			\item[(ii)] \S 18, \textbf{Corollary} 18.2: The relationship of terminating finitely with respect to the transfinitely defined ordinal relation `$>_0$' over the set of finite ordinals does not entail the relationship of terminating finitely with respect to the finitarily defined natural number relation `$>$' over the set of natural numbers;
			
			\item[(iii)] \S 18.a, \textbf{Theorem} 18.4: The subsystem ACA$_0$ of second-order arithmetic is \textit{not} a conservative extension of PA.	
		\end{enumerate}
		}, that the `categoricity' of PA reflects the circumstance that, contrary to inherited paradigms, the subsystem ACA$_0$ of second-order arithmetic is \textit{not} a conservative extension of PA.
	
	\vspace{+1ex}
	Consequently\footnote{See \cite{An21}, \S 13.E: \textit{Recognising the strengths and limitations of ZF and PA}.}, we may need to recognise explicitly that \textit{evidence-based} reasoning:
	
	\begin{itemize}
		\item[(a)] restricts the ability of highly expressive mathematical languages, such as the first-order Zermelo-Fraenkel Set Theory ZF, to \textit{categorically} communicate abstract concepts\footnote{Corresponding to Lakoff and N\'{u}\~{n}ez's conceptual metaphors in \cite{LR00}; see also \cite{An21}, \S 25: \textit{The significance of \textit{evidence-based} reasoning for Cognitive Science}.} such as those involving Cantor's first limit ordinal $\omega$\footnote{See \cite{LR00}, Preface, p.\textit{xii-xiii}: ``\textit{How can human beings understand the idea of actual infinity?}"};		
	\end{itemize}
	
	\noindent and:
	
	\begin{itemize}
		\item[(b)] restricts the ability of effectively communicating mathematical languages, such as the first-order Peano Arithmetic PA, to \textit{well-define} infinite concepts such as $\omega$\footnote{See \S \cite{An21}, 17.A.a: \textit{We cannot force PA to admit a transfinite ordinal}.}.
	\end{itemize}
\end{quote}

\subsection{G\"{o}del's `\textit{undecidable}' formula $[\neg(\forall x)R(x)]$ is provable in PA}
\label{sec:7}

Moreover, the Provability Theorem for PA further entails\footnote{See \cite{An16}, \textbf{Lemma 8.1}: If $\mathcal{I}_{PA(\mathbb{N},\ M)}$ defines a model of PA over $\mathbb{N}$, then there is a PA formula $[F]$ which is algorithmically \textit{verifiable} as always true over $\mathbb{N}$ under $\mathcal{I}_{PA(\mathbb{N},\ M)}$ even though $[F]$ is not PA-provable..} that:

\begin{lemma}
\label{sec:1.03.lem.2}
	If $\mathcal{I}_{PA(\mathbb{N},\ M)}$ defines a model of PA over $\mathbb{N}$, then there is a PA formula $[F]$ which is algorithmically \textit{verifiable} as always true over $\mathbb{N}$ under $\mathcal{I}_{PA(\mathbb{N},\ M)}$  even though $[F]$ is not PA-provable.
\end{lemma}

\begin{proof} G\"{o}del has shown how to construct a quantifier-free arithmetical formula with a single variable---say $[R(x)]$\footnote{G\"{o}del refers to the formula $[R(x)]$ only by its G\"{o}del number $r$ (\cite{Go31}, p.25(12)).}---such that, if PA is consistent, then $[R(x)]$ is not PA-provable\footnote{G\"{o}del's aim in \cite{Go31} was to show that $[(\forall x)R(x)]$ is not P-provable; by Generalisation it follows, however, that $[R(x)]$ is also not P-provable.}, but $[R(n)]$ is instantiationally (i.e., numeral-wise) PA-provable for any given PA numeral $[n]$. Since PA is consistent by \S \ref{sec:pa.axms.true.alg.cmp}, Theorem \ref{sec:6.thm.2}, for any given numeral $[n]$, G\"{o}del's primitive recursive relation $xB \ulcorner [R(n)] \urcorner$\footnote{Where $\ulcorner [R(n)] \urcorner$ denotes the G\"{o}del-number of the PA formula $[R(n)]$.} must hold for some $x$. The lemma follows.\hfill $\Box$
\end{proof}

The Provability Theorem for PA also immediately entails the following, \textit{finitary}, arithmetical consequences\footnote{Corresponding to \cite{An16}, Corollaries 8.2 to 8.4.}, which contradict current, implicitly \textit{non-finitary} and \textit{essentially} set-theoretical, paradigms that appeal to G\"{o}del's \textit{informal}\footnote{And misleading, as detailed in \cite{An21}, \S 14, \textit{G\"{o}del 1931 in hindsight}.} interpretation of his own \textit{formal} reasoning in \cite{Go31}:

\begin{corollary}
\label{sec:6.3.cor.1}
	The PA formula $[\neg(\forall x)R(x)]$ defined in Lemma \ref{sec:1.03.lem.2} is PA-provable. \hfill $\Box$
\end{corollary}

\begin{corollary}
\label{sec:6.3.cor.1.1}
	In any \textit{well-defined} model of PA, G\"{o}del's arithmetical formula $[R(x)]$ interprets as an algorithmically verifiable, but not algorithmically computable, tautology over $\mathbb{N}$.
\end{corollary}

\begin{proof} G\"{o}del has shown that $[R(x)]$\footnote{G\"{o}del refers to the formula $[R(x)]$ only by its G\"{o}del number $r$; \cite{Go31}, p.25, eqn.12.} always interprets as an algorithmically \textit{verifiable} tautology over $\mathbb{N}$\footnote{\cite{Go31}, p.26(2): ``$(n)\neg(nB_{\kappa}(17Gen\ r))$ holds"}. By Corollary \ref{sec:6.3.cor.1} $[R(x)]$ is not algorithmically computable as always true in $\mathbb{N}$.\hfill $\Box$
\end{proof}

\begin{corollary}
\label{sec:6.3.cor.2}
	PA is \textit{not} $\omega$-consistent.
\end{corollary}

\begin{proof} G\"{o}del has shown that if PA is consistent, then $[R(n)]$ is PA-provable for any given PA numeral $[n]$\footnote{\cite{Go31}, p.26(2).}. By Corollary \ref{sec:6.3.cor.1} and the definition of $\omega$-consistency, if PA is consistent then it is \textit{not} $\omega$-consistent.\hfill $\Box$
\end{proof}

\begin{quote}
\footnotesize
	\textbf{Comment}: We prove Corollary \ref{sec:6.3.cor.2} independently in \cite{An21}, \S 12.B.f, Theorem 12.6. We note that this conclusion is contrary to accepted dogma. See, for instance, Martin Davis' remarks that:

	\vspace{+1ex}
	 ``\ldots there is no equivocation. Either an adequate arithmetical logic is $\omega$-inconsistent (in which case it is possible to prove false statements within it) or it has an unsolvable decision problem and is subject to the limitations of G\"{o}del's incompleteness theorem". \\ \textit{\tiny{\ldots Davis: \cite{Da82}, p.129(iii).}}
	 
	 \vspace{+1ex}
	 \textbf{Comment}: We note that Rosser's appeal to his Rule C, in his proof \cite{Ro36} claiming that simple consistency suffices for establishing a `formally undecidable' arithmetical formula (which involves an existential quantifier) in $P$, implicitly entails that $P$ is $\omega$-consistent since, by \S \ref{sec:intro.prov.thm.pvnp}, Definition \ref{def:ap}:
	 
	 \begin{quote}								
	 	\begin{corollary}
	 	\label{lem:ruleC.equiv.omega.cons}
	 		Rosser's Rule $C$ \textit{entails} G\"{o}del's $\omega$-consistency.
	 		
	 		\index{G\"{o}del!$\omega$-consistency}%
	 		\index{$\omega$-consistency!G\"{o}del}%
	 	\end{corollary}
	 	
	 	\begin{proof} If $P$ is simply consistent, the introduction of an \textit{unspecified} $P$-term into the formal reasoning under Rule C entails Aristotle's particularisation in any interpretation of $P$, which in turn entails that $P$ is $\omega$-consistent. The corollary follows. \hfill $\Box$
	 	\end{proof}
	 \end{quote}
	 
	 \textbf{Rosser's Rule C} (\textit{Excerpted from Mendelson \cite{Me64}, p.73-74, \S 7, Rule C; see also \cite{Ro53}, pp.127-130.})
	 
	 \begin{quote}
	 	\noindent ``It is very common in mathematics to reason in the following way. Assume that we have proved a wf of the form $(Ex)\mathcal{A}(x)$. Then, we say, let $b$ be an object such that $\mathcal{A}(b)$. We continue the proof, finally arriving at a formula which does not involve the arbitrarily chosen element $b$. \ldots
	 	
	 	\vspace{+1ex}
	 	\noindent In general, any wf which can be proved using arbitrary acts of choice, can also be proved without such acts of choice. We shall call the rule which permits us to go from $(Ex)\mathcal{A}(x)$ to $\mathcal{A}(b)$, Rule C (``C" for ``choice")."
	 \end{quote}
\end{quote}

	\section{The significance of the Provability Theorem for PA for the P$v$NP problem}
	\label{sec:significance.prov.thm.pvnp}
	
	From the \textit{evidence-based} perspective of this investigation (expressed explicitly as a Complementarity Thesis in \cite{An21}, \S 1, Thesis 1), the significance of the Provability Theorem for PA (\S \ref{sec:6.3.a}, Theorem \ref{sec:6.3.thm.1}) for the P$v$NP problem is that (compare \S \ref{sec:7}, Corollary \ref{sec:6.3.cor.1.1}):
	
	\begin{theorem}[First Tautology Theorem]
	\label{thm:first.tautology.thm}
		There is no Turing-machine, deterministic or non-deterministic, that can evidence G\"{o}del's arithmetical relation \(R^{*}(x)\)---when treated as a Boolean function---as a tautology.
	\end{theorem}
	
	\begin{proof}	
		In his seminal 1931 paper \cite{Go31}, G\"{o}del has constructed a PA-formula $[(\forall x)R(x)]$ such that $[R(n)]$ is \textit{without} quantifiers and PA-provable for any \textit{specified} PA-numeral $[n]$. Hence, under any \textit{well-defined} interpretation of PA over $\mathbb{N}$, $[R(x)]$ interprets as a tautological (\textit{quantifier-free}) arithmetical relation \(R^{*}(x)\) which is true for any \textit{specified} natural number \(n\). 
		
		\vspace{+1ex}
		However, since the corresponding PA-formula [\(R(x)\)]\footnote{G\"{o}del defines, and refers to, this formula by its G\"{o}del-number \(r\) (cf.\ \cite{Go31}, p25, eqn.12).} is not PA-provable (cf.\ \cite{Go31}, p25(1))---and it follows from the Provability Theorem for PA that any computational problem which can be solved by an algorithm that defines a \textit{deterministic} Turing-machine, can also be solved by some algorithm which defines a \textit{non-deterministic} Turing-machine, and vice versa---there can be no Turing-machine, deterministic or non-deterministic, that could evidence \(R^{*}(x)\) as a tautology (\textit{i.e., as true for any \textit{specified} natural number} \(n\)). \hfill $\Box$
	\end{proof}
	
	\begin{quote}
		\footnotesize
		\textbf{Comment}: By Generalisation\footnote{\textit{Generalisation in PA}: [\((\forall x)A\)] follows from [\(A\)].}, stating that the PA-formula [\(R(x)\)] \textit{is not PA-provable} is equivalent to stating that the PA-formula [\((\forall x)R(x)\)]\footnote{G\"{o}del defines, and refers to, this formula by its G\"{o}del-number \(17Gen\ r\) (cf.\ \cite{Go31}, p25, eqn.13).} \textit{is not PA-provable}; the latter is what G\"{o}del actually proved in \cite{Go31}.
	\end{quote}
	
	\noindent We also have, further, that:
	
	\begin{theorem}[Second Tautology Theorem]
		\label{thm:second.tautology.thm}
		G\"{o}del's arithmetical relation \(R^{*}(x)\) is algorithmically \textit{verifiable} as a tautology.
	\end{theorem}
	
	\begin{proof}
		G\"{o}del has defined a primitive recursive relation, \(xB_{_{PA}}y\) that holds if, and only if, \(y\) is the G\"{o}del-number of a PA-formula with one variable, say [\(R\)], and \(x\) the G\"{o}del-number of a PA-proof of [\(R\)] (\cite{Go31}, p22, dfn.\ 45).
		
		\vspace{+1ex}
		Since every primitive recursive relation is Turing-computable (\textit{when treated as a Boolean function}), \(xB_{_{PA}}y\) defines a Turing-machine TM\(_{B}\) that halts on any \textit{specified} natural number values of \(x\) and \(y\).
		
		\vspace{+1ex}
		Now, if \(g_{[R(1)]}\), \(g_{[R(2)]}\), \ldots are the G\"{o}del-numbers of the PA-formulas [\(R(1)\)], [\(R(2)\)], \dots, it follows that, for any \textit{specified} natural number \(n\), when  the natural number value \(g_{[R(n)]}\) is input for \(y\), the Turing-machine TM\(_{B}\) must halt for some value of \(x\)---which is the G\"{o}del-number of some PA-proof of [\(R(n)\)]---since G\"{o}del has shown (\cite{Go31}, p25(1)) that [\(R(n)\)] is PA-provable for any \textit{specified} numeral [\(n\)].
		
		\vspace{+1ex}
		\noindent Hence \(R^{*}(n)\) is algorithmically \textit{verifiable} as true for any \textit{specified} natural number \(n\). \hfill \(\Box\)  
	\end{proof}

	\subsection{Interpreting the standard definition of the P$v$NP problem \textit{finitarily}}
	\label{sec:pvnp.problem}
	
	We note that the \textit{standard} definition of the classes P and NP is the one provided by Stephen Cook to the Clay Mathematical Institute in a 2000 paper, \cite{Cook}, which current computational complexity paradigms accept as the definitive description of the P$v$NP problem; where Cook admits a number-theoretic function $F$---viewed set-theoretically as extensionally defining (and being defined by) a unique subset $L$ of the set $\Sigma^{*}$ of finite strings over some non-empty finite alphabet set $\Sigma$---in P if, and only if, some deterministic Turing machine TM accepts $L$ and runs in polynomial time:
	
	\begin{quote}
		\footnotesize
		``The computability precursors of the classes $P$ and $NP$ are the classes of decidable and c.e. (computably enumerable) languages, respectively. We say that a language $L$ is c.e. i.e. (or semi-decidable) iff $L = L(M)$ for some Turing machine $M$. We say that $L$ is decidable iff $L = L(M)$ for some Turing machine $M$ which satisfies the condition that $M$ halts on all input strings $w$. \ldots
		
		\vspace{+.5ex}
		Thus the problem Satisfiability is: Given a propositional formula $F$, determine whether $F$ is satisfiable. To show that this is in $NP$ we define the polynomial-time checking relation $R(x, y)$, which holds iff $x$ codes a propositional formula $F$ and $y$ codes a truth assignment to the variables of $F$ which makes $F$ true." \\ \textit{\tiny{\ldots Cook: \cite{Cook}.}}
	\end{quote}
	
	However, we shall prefer\footnote{For reasons detailed in \cite{An21}, \S 4.B: \textit{An implicit ambiguity in the `official' definition of P}.} to interpret number-theoretic functions and relations over an infinite domain $\mathbb{D}$ \textit{finitarily}; as \textit{well-defining} deterministic algorithms (\S \ref{sec:intro.prov.thm.pvnp}, Definition \ref{def:det.alg}) that, for any \textit{specified} sequence of permissible values to the variables in the function/relation, \textit{evidence} how the function/relation is to be evaluated---and whether, or not, the result of such evaluation yields a value (or values)---in the domain $\mathbb{D}$.
	
	\vspace{+1ex}
	We shall not assume---as is the case in \textit{non-finitary} Cantorian set theories such as the first-order set Theory ZF, or the second-order Peano Arithmetic ACA$_{0}$\footnote{See \cite{An21}, \S 18.A: \textit{The subsystem ACA$_0$}.}---that the evaluations always determine a completed infinity (set) which can be referred to as a unique mathematical constant that identifies the function/relation in a mathematical language (or its interpretation) \textit{outside} of the set theory in which the function/relation is defined.
	
	\begin{quote}
		\footnotesize
		\textbf{Comment}: In other words we do not ignore Thoralf Skolem's cautionary remarks against inviting paradox by conflating entailments of formal systems under different interpretations (such as the two \textit{well-defined} interpretations of PA in \S \ref{sec:alg.ver.int.pa.n} and \S \ref{sec:alg.cmp.int.pa}), or over different domains\footnote{See, for instance, Goodstein's \textit{paradox} in \cite{An21}, \S 18,  \textbf{Theorem 18.1}: Goodstein's sequence $G_{o}(m_{o})$ over the finite ordinals in any putative model $\mathbb{M}$ of ACA$_{_{0}}$ \textit{terminates} with respect to the ordinal inequality `$>_{o}$' even if Goodstein's sequence $G(m)$ over the natural numbers \textit{does not terminate} with respect to the natural number inequality `$>$' in $\mathbb{M}$.}; where we note that, in a 1922 address delivered in Helsinki before the Fifth Congress of Scandinavian Mathematicians, Skolem improved upon both the argument and statement of L\"{o}wenheim's 1915 theorem (\cite{Lo15}, p.235, Theorem 2)---subsequently labelled as the (downwards) L\"{o}wenheim-Skolem Theorem (\cite{Sk22}, p.293):		
			
		\begin{itemize}
			\item[ ] \textbf{(Downwards) L\"{o}wenheim-Skolem Theorem} (\cite{Lo15}, p.245, Theorem 6; \cite{Sk22}, p.293): If a first-order proposition is satisfied in any domain at all, then it is already satisfied in a denumerably infinite domain.
		\end{itemize}
		
		\noindent Skolem then drew attention to a:
		
		\begin{enumerate}
			\footnotesize
			\item[ ] \textbf{Skolem's (apparent) paradox}: ``\ldots peculiar and apparently paradoxical state of affairs. By virtue of the axioms we can prove the existence of higher cardinalities, of higher number classes, and so forth. How can it be, then, that the entire domain $B$ can already be enumerated by means of the finite positive integers? The explanation is not difficult to find. In the axiomatization, ``set" does not mean an arbitrarily defined collection; the sets are nothing but objects that are connected with one another through certain relations expressed by the axioms. Hence there is no contradiction at all if a set $M$ of the domain $B$ is non-denumerable in the sense of the axiomatization; for this means merely that \textit{within} $B$ there occurs no one-to-one mapping $\Phi$ of $M$ onto $Z_{o}$ (Zermelo's number sequence). Nevertheless there exists the possibility of numbering all objects in $B$, and therefore also the elements of $M$, by means of the positive integers; of course such an enumeration too is a collection of certain pairs, but this collection is not a ``set" (that is, it does not occur in the domain $B$)." \\ \textit{\tiny{\ldots Skolem: \cite{Sk22}, p.295.}}
		\end{enumerate}	
	\end{quote}	

	We shall, instead, address the P$v$NP problem here from the \textit{logical} perspective of the paper \cite{Ra02} presented to ICM 2002 by Ran Raz, where he notes that:
	
	\begin{quote}
		\footnotesize
		``A Boolean formula $f(x_{1}, \ldots, x_{n})$ is a tautology if $f(x_{1}, \ldots, x_{n}) = 1$ for every $x_{1}, \ldots, x_{n}$. A Boolean formula $f(x_{1}, \ldots, x_{n})$ is unsatisfiable if $f(x_{1}, \ldots, x_{n}) = 0$ for every $x_{1}, \dots, x_{n}$. Obviously, $f$ is a tautology if and only if $\neg f$ is unsatisfiable.
		
		Given a formula $f(x_{1}, \ldots, x_{n})$, one can decide whether or not $f$ is a tautology by checking all the possibilities for assignments to $x_{1}, \ldots, x_{n}$. However, the time needed for this procedure is exponential in the number of variables, and hence may be exponential in the length of the formula $f$.
		
		\ldots P$\neq$NP is the central open problem in complexity theory and one of the most important open problems in mathematics today. The problem has thousands of equivalent formulations. One of these formulations is the following:
		
		Is there a polynomial time algorithm $\mathcal{A}$ that gets as input a Boolean formula $f$ and outputs 1 if and only if $f$ is a tautology?
		
		P$\neq$NP states that there is no such algorithm." \\ \textit{\tiny{\ldots Raz: \cite{Ra02}.}}
	\end{quote}
	
	\noindent Defining SAT as the class of Boolean tautologies, the corresponding SAT problem for PA is then:
	
	\begin{definition}[SAT$_{PA}$]
		\label{def:sat}
		The Boolean satisfiability problem (SAT$_{PA}$) for the first-order Peano Arithmetic PA is, given an arithmetical formula without quantifiers, whether, or not, a polynomial time algorithm can evidence it as a tautology.
	\end{definition}

	\subsubsection{SAT \(\notin\) P \textit{and} SAT \(\notin\) NP}
	\label{sec:lgcl.prf.p.neq.np}
	
	Clearly, the issue of whether, or not, there is a \textit{polynomial time} `algorithm $\mathcal{A}$ that gets as input a Boolean formula $f$ and outputs 1 if, and only if, $f$ is a tautology' is meaningful \textit{only} if we can \textit{evidence} that there \textit{is}, in fact, an `algorithm $\mathcal{A}$ that gets as input a Boolean formula $f$ and outputs 1 if, and only if, $f$ is a tautology'.
	
	\vspace{+1ex}
	So, if the G\"{o}delian relation \(R^{*}(x)\) defined in \S \ref{sec:significance.prov.thm.pvnp}, Theorem \ref{thm:first.tautology.thm} is algorithmically \textit{verifiable}, but not algorithmically \textit{computable}, as a tautology---hence not recognisable as a tautology by any Turing-machine, whether \textit{deterministic} or \textit{non-deterministic}---then it is trivially true \textit{logically} that:
	
	\begin{theorem}
		\label{thm:sat.prf.pvnp}
		SAT \(\notin\) P and SAT \(\notin\) NP.
	\end{theorem}
	
	\begin{proof}
		By \S \ref{sec:significance.prov.thm.pvnp}, Theorem \ref{thm:first.tautology.thm} and Theorem \ref{thm:second.tautology.thm}, \textit{no} Turing-machine, whether \textit{deterministic} or \textit{non-deterministic}, can evidence whether or not G\"{o}del's arithmetical relation $R^{*}(x)$ is a tautology. Hence SAT$_{PA}$ \(\notin\) P and SAT$_{PA}$ \(\notin\) NP. Since SAT$_{PA}$ $\subset$ SAT, the Theorem follows. \hfill $\Box$
	\end{proof}

	\section{Appendix: An implicit ambiguity in the `official' definition of P}
	\label{sec:ambgty.offcl.dfn.p}
	
	We note that, in a 2009 survey \cite{Frt09} of the status of the P$v$NP problem, Lance Fortnow despairs that `we have little reason to believe we will see a proof separating P from NP in the near future' since `[n]one of us truly understand the P versus NP problem':
	
	\begin{quote}
		\footnotesize
		``\ldots in the mid-1980's, many believed that the quickly developing area of circuit complexity would soon settle the P versus NP problem, whether every algorithmic problem with efficiently verifiable solutions have efficiently computable solutions. But circuit complexity and other approaches to the problem have stalled and we have little reason to believe we will see a proof separating P from NP in the near future.
		
		\ldots As we solve larger and more complex problems with greater computational power and cleverer algorithms, the problems we cannot tackle begin to stand out. The theory of NP-completeness helps us understand these limitations and the P versus NP problems begins to loom large not just as an interesting theoretical question in computer science, but as a basic principle that permeates all the sciences.
		
		\ldots None of us truly understand the P versus NP problem, we have only begun to peel the layers around this increasingly complex question." \\ \textit{\tiny{\ldots Fortnow: \cite{Frt09}.}}
	\end{quote}
	
	In this investigation we shall argue that Fortnow's pessimism reflects the circumstance that standard, set-theoretical, interpretations---such as the following\footnote{See also \cite{Mor12}.}---of the formal definitions of the classes P and NP in \cite{Cook} can be seen to admit an implicit ambiguity:
	
	\begin{quote}
		\footnotesize
		``The computability precursors of the classes $P$ and $NP$ are the classes of decidable and c.e. (computably enumerable) languages, respectively. We say that a language $L$ is c.e. i.e. (or semi-decidable) iff $L = L(M)$ for some Turing machine $M$. We say that $L$ is decidable iff $L = L(M)$ for some Turing machine $M$ which satisfies the condition that $M$ halts on all input strings $w$." \\ \textit{\tiny{\ldots Cook: \cite{Cook}.}}
		
		\begin{quote}
			\footnotesize
			\textbf{Comment}: For instance, it is not clear from the above whether (a) $S \in P$ iff $S$ is decidable and $S \in NP$ iff $S$ is c.e, in which case the separation between the two classes would be qualitative; or whether (b) both $P$ and $NP$ are decidable classes, in which case (following contemporary wisdom) the separation between the two classes can be assumed to be only quantitative.
		\end{quote}
	\end{quote}
	
	Specifically, from the \textit{evidenced} based perspective of this investigation concerning the relative strengths and limitations of first order set theories and first order arithmetics\footnote{As reflected in the \textit{Complementarity Thesis} in \cite{An21} (\S 1, Thesis 1), and argued in \cite{An21}, \S 13.E: \textit{Recognising the strengths and limitations of ZF and PA}.}, set-theoretical interpretations of the P$v$NP problem are \textit{essentially} unable to recognise that the assignment of satisfaction and truth values to number-theoretic formulas, under a \textit{well-defined} interpretation, can be defined in two, distinctly different, \textit{evidence-based} ways\footnote{The distinction is explicitly introduced, and its significance in establishing a finitary proof of consistency for the first order Peano Arithmetic PA highlighted, by Theorem 6.8, p.41, in \cite{An16} (see also \S \ref{sec:alg.cmp.int.pa}, Theorem \ref{sec:6.thm.2}).}:
	
	\vspace{+1ex}
	(a) in terms of algorithmic \textit{verifiability} (see \S \ref{sec:revisiting}, Definition \ref{sec:2.def.1});
	
	\begin{quote}
		\footnotesize
		\textbf{Comment}: It immediately follows from this definition that a number-theoretical formula $F$ is algorithmically \textit{verifiable} under an interpretation (and should therefore be defined in NP) if, and only if, we can define a checking relation $R(x, y)$\footnote{If $F$ is a formula of the first order Peano Arithmetic PA, the existence of such a checking relation is assured by Theorem 5.1, p.38, in \cite{An16} (see also \S \ref{sec:alg.ver.int.pa.n}, Theorem \ref{sec:5.5.lem.1}).}---where $x$ codes a propositional formula $F$ and $y$ codes a truth assignment to the variables of $F$---such that, for any given natural number values $(m, n)$, there is a deterministic algorithm which will finitarily decide whether or not $R(m, n)$ holds over the domain $\mathbb{N}$ of the natural numbers.
	\end{quote}
	
	(b) in terms of algorithmic \textit{computability} (see \S \ref{sec:revisiting}, Definition \ref{sec:2.def.2}).
	
	\begin{quote}
		\footnotesize
		\textbf{Comment}: It immediately follows from this definition that a number-theoretical formula $F$ is algorithmically \textit{computable} under an interpretation (and should therefore be defined in P) if, and only if, we can define a checking relation $R(x, y)$\footnote{If $F$ is a PA formula, the existence of such a checking relation is assured by Theorem 6.1, p.40, in \cite{An16} (see also \S \ref{sec:alg.cmp.int.pa}, Theorem \ref{sec:5.5.lem.2}).}---where $x$ codes a propositional formula $F$ and $y$ codes a truth assignment to the variables of $F$---such that there is a deterministic algorithm which, for any given natural number values $(m, n)$, will finitarily decide whether or not $R(m, n)$ holds over the domain $\mathbb{N}$ of the natural numbers.
	\end{quote}
	
	Consequently, standard, set-theoretical, interpretations of the formal definitions of the classes P and NP which do not admit the relative strengths and limitations of first order set theories and first order arithmetics, would be prone to implicitly assuming that every propositional formula which is algorithmically \textit{verifiable} is necessarily algorithmically \textit{computable}.
	
	\vspace{+1ex}
	It would then follow that the differentiation between the classes P and NP is only quantitative, and can therefore be adequately expressed in terms of computational complexity; i.e., whether or not the class P can be defined as consisting of all, and only, those problems that can be solved in \textit{polynomial time} by a \textit{deterministic} Turing machine, and the class NP as consisting of all, and only, those problems that can be solved in \textit{polynomial time} by a \textit{non-deterministic} Turing machine.
	
	\vspace{+1ex}
	However, we shall argue that---since the two concepts \S \ref{sec:ambgty.offcl.dfn.p}(a) and \S \ref{sec:ambgty.offcl.dfn.p}(b) are \textit{well-defined}, and there are classically defined arithmetic formulas---such as G\"{o}del's `undecidable' formula $[R(x)]$\footnote{Which G\"{o}del defines and refers to only by its G\"{o}del number $r$ in \cite{Go31}, p.25, eqn.12.}---which are algorithmically \textit{verifiable} but not algorithmically \textit{computable} (see \cite{An16}, Corollary 8.3, p.42; also \S \ref{sec:7}, Corollary \ref{sec:6.3.cor.1.1}), the differentiation between the classes P and NP is also qualitative, and cannot be adequately expressed in terms of only computational complexity.

	\subsection{The P$v$NP Separation Problem}
	\label{sec:pnp.sprtn.prblm}
	
	In his 2009 survey \cite{Frt09}, Fortnow describes the P$v$NP problem informally as follows:
	
	\begin{quote}
		\footnotesize
		``In 1965, Jack Edmonds \ldots suggested a formal definition of ``efficient computation" (runs in time a fixed polynomial of the input size). The class of problems with efficient solutions would later become known as P for ``Polynomial Time".
		
		\ldots But many related problems do not seem to have such an efficient algorithm.
		
		\ldots The collection of problems that have efficiently verifiable solutions is known as NP (for ``Nondeterministic Polynomial-Time" \dots).
		
		So P=NP means that for every problem that has an efficiently verifiable solution, we can find that solution efficiently as well.
		
		\ldots If a formula $\phi$ is not a tautology, we can give an easy proof of that fact by exhibiting an assignment of the variables that makes $\phi$ false. But if \ldots there are no short proofs of tautology that would imply P$\neq$NP." \\ \textit{\tiny{\ldots Fortnow: \cite{Frt09}.}}
	\end{quote}
	
	This suggests the following, implicitly set-theoretical, formulation of the P$v$NP Separation Problem:
	
	\begin{query}[Efficient P$v$NP Separation]
		\label{qry:effcnt.pnp.sprtn.prblm}
		Is there an arithmetical formula $F$ that is \textit{efficiently} verifiable and not \textit{efficiently} computable? 
	\end{query}
	
	From the \textit{evidence-based} perspective of this investigation we shall, however, consider a more precise formulation in arithmetic.
	
	\vspace{+1ex}
	In other words, we shall avoid the ambiguity---in the meaning of Edmonds' concept of `efficient'---which is admitted by asymmetrically defining `efficient computation' as computation by a \textit{deterministic} Turing machine in \textit{polynomial} time, and `efficient verification' as computation by a \textit{non-deterministic} Turing machine in \textit{polynomial time}.
	
	\vspace{+1ex}
	We shall, instead, define Edmonds' `efficient \textit{computation}' as `algorithmic \textit{computation}', and `efficiently \textit{verifiable}' as `algorithmically \textit{verifiable}'; whence an affirmative answer to Query \ref{qry:effcnt.pnp.sprtn.prblm} would entail, and be entailed by, an affirmative answer to:
	
	\begin{query}[Algorithmic P$v$NP Separation]
		\label{qry:algrthmc.pnp.sprtn.prblm}
		Is there an arithmetical formula $F$ that is \textit{algorithmically} verifiable but not \textit{algorithmically} computable? 
	\end{query}
	
	We shall now show that Query \ref{qry:algrthmc.pnp.sprtn.prblm} not only removes the ambiguity in the standard, set-theoretical, asymmetrical definitions of the classes P and NP, but it also admits of an affirmative answer.
	
	\vspace{+1ex}
	We shall first show how G\"{o}del's $\beta$-function\footnote{See \cite{Go31}, Theorem VII, Lemma 1; also \cite{An21}, \S 15.A: \textit{G\"{o}del's $\beta$-function}.} uniquely corresponds each classically defined real number to an algorithmically \textit{verifiable} arithmetical formula.
	
	\vspace{+1ex}
	Since classical theory admits the existence of real numbers that are not algorithmically \textit{computable}\footnote{As detailed in \cite{Tu36}.}, we shall conclude that classical theory must also admit the existence of arithmetical formulas that are algorithmically \textit{verifiable} but not algorithmically \textit{computable}.
	
	\vspace{+1ex}
	We note, first, that every \textit{atomic} number-theoretical formula is algorithmically \textit{verifiable}\footnote{An immediate consequence of \cite{Tu36}.}; further, by Tarski's definitions\footnote{On the inductive assignment of satisfaction and truth values to the formulas of a formal language under an interpretation, as detailed in \cite{Ta35}.}, the algorithmic \textit{verifiability} of the \textit{compound} formulas of a formal language (which contain additional logical constants) can be inductively defined---under a \textit{well-defined} interpretation---in terms of the algorithmic \textit{verifiability} of the interpretations of the atomic formulas of the language\footnote{See, for instance, \cite{An16}, \S 3: \textit{Reviewing Tarski's inductive assignment of truth-values under an interpretation}; also \S \ref{sec:rvwng.trski.ind.def}}.
	
	\vspace{+1ex}
	In particular, by \S \ref{sec:alg.ver.int.pa.n}, Corollary \ref{sec:5.4.lem.5} and Theorem \ref{sec:5.4.thm.2}\footnote{See also \cite{An16}, \S 5, Theorems 5.6 and 5.7.}, the formulas of the first order Peano Arithmetic PA are decidable under the \textit{weak}, algorithmically \textit{verifiable}, interpretation $\mathcal{I}_{PA(\mathbb{N},\ SV)}$ of PA over the domain $\mathbb{N}$ of the natural numbers if, and only if, they are algorithmically \textit{verifiable} under the interpretation.
	
	\vspace{+1ex}
	Similarly, every \textit{atomic} number-theoretical formula is algorithmically \textit{computable}\footnote{An immediate consequence of \cite{Tu36}.}; further, by Tarski's definitions, the algorithmic \textit{computability} of the \textit{compound} formulas of a formal language (which contain additional logical constants) can be inductively defined---under a \textit{well-defined} interpretation---in terms of the algorithmic \textit{computability} of the interpretations of the atomic formulas of the language\footnote{See, for instance, \cite{An16}, \S 3; also \S \ref{sec:rvwng.trski.ind.def}}.
	
	\vspace{+1ex}
	In this case, however, the PA-formulas are \textit{always} decidable under the \textit{strong}, finitary, algorithmically \textit{computable} interpretation $\mathcal{I}_{PA(\mathbb{N},\ SC)}$ of PA over $\mathbb{N}$, since PA is categorical with respect to algorithmic \textit{computability}\footnote{By \cite{An16}, Corollary 7.2: \textit{PA is categorical with respect to algorithmic \textit{computability}.}; see also \S \ref{sec:6.3.a}, Corollary \ref{cor:pa.categorical}.}.
	
	\vspace{+1ex}
	We also note that, by \cite{An16}, Theorem 2.1\footnote{See also \cite{An21}, \S 7.G, Theorem 7.2: \textit{There are \textit{well-defined} number theoretic functions that are algorithmically \textit{verifiable} but not algorithmically \textit{computable}}.}, there are algorithmically \textit{verifiable} number theoretical formulas which are not algorithmically \textit{computable}.
	
	\vspace{+1ex}
	We note that algorithmic \textit{computability} implies the existence of a deterministic algorithm that can \textit{finitarily} decide the truth/falsity of each proposition in a well-defined denumerable sequence of propositions\footnote{Which is why (see \S \ref{sec:pa.axms.true.alg.cmp}, \textit{Poincar\'{e}-Hilbert debate} (ii)) the PA Axiom Schema of Finite Induction \textit{can} be \textit{finitarily} verified as true (see \S \ref{sec:pa.axms.true.alg.cmp}, Lemma \ref{sec:6.lem.2}) under the \textit{strong}, finitary, algorithmically \textit{computable} interpretation $\mathcal{I}_{PA(\mathbb{N},\ SC)}$ of PA, over $\mathbb{N}$, with respect to `truth' as defined by the algorithmically \textit{computable} formulas of PA.}, whereas algorithmic \textit{verifiability} does not imply the existence of a deterministic algorithm that can \textit{finitarily} decide the truth/falsity of each proposition in a well-defined denumerable sequence of propositions\footnote{Which is why, in this case (see \S \ref{sec:pa.axms.true.alg.cmp}, \textit{Poincar\'{e}-Hilbert debate} (i)), the PA Axiom Schema of Finite Induction \textit{cannot} be \textit{finitarily} verified as true---but only algorithmically \textit{verified} as true (see \S \ref{sec:pa.axms.true.alg.ver}, Lemma \ref{sec:5.4.lem.2})---under the \textit{weak}, standard (see \S \ref{sec:alg.ver.int.pa.n}), algorithmically \textit{verifiable} interpretation $\mathcal{I}_{PA(\mathbb{N},\ SV)}$ of PA, over $\mathbb{N}$, with respect to `truth' as defined by the algorithmically \textit{verifiable} formulas of PA.}.
	
	\index{algorithmic!computability}%
	\index{computability!algorithmic}%
	
	\vspace{+1ex}
	From the point of view of a \textit{finitary} mathematical philosophy, the significant difference between the two concepts could be expressed by saying that we may treat the decimal representation of a real number as corresponding to a physically measurable limit---and not only to a mathematically definable limit---if and only if such representation is definable by an algorithmically \textit{computable} function.\footnote{The significance of this for the natural sciences is highlighted in \cite{An21}, \S 19.C: \textit{Mythical `set-theoretical' limits of fractal constructions}.}

	\subsubsection{An arithmetical perspective on the P$v$NP Separation Problem}
	\label{sec:arthmtcl.perspctv.pvnp.sprtn.prblm}
	
	We finally argue that G\"{o}del's $\beta$-function (see \S \ref{gdl.beta.fn}: \textit{G\"{o}del's $\beta$-function}) entails:
	
	\begin{theorem}
		\label{thm:pvnp.alg.ver.nt.alg.cmp}
		There is an arithmetical formula that is algorithmically verifiable, but not algorithmically computable, under any \textit{evidence-based} interpretation of PA.
	\end{theorem}
	
	\begin{proof}  Let $\{r(n)\}$ be the denumerable sequence defined by the denumerable sequence of digits in the decimal expansion $\sum_{i=1}^{\infty}r(i).10^{-i}$ of a putatively \textit{well-defined} real number $\mathbb{R}$ in the interval $0 < \mathbb{R} \leq 1$.
		
		\vspace{+1ex}
		\noindent By \S \ref{gdl.beta.fn}, Lemma \ref{sec:ch.2.2.2.lem.1}, for any \textit{specified} natural number $k$, there are natural numbers $b_{k}, c_{k}$ such that, for any $1 \leq n \leq k$:
		
		\vspace{+1ex}
		$\beta(b_{k}, c_{k}, n) = r(n)$.
		
		\vspace{+1ex}
		\noindent By \S \ref{gdl.beta.fn}, Lemma \ref{sec:ch.2.2.1.lem.1}, $\beta(x_{1}, x_{2}, x_{3})$ is strongly represented in PA by $[Bt(x_{1}, x_{2}, x_{3}, x_{4})]$ such that, for any $1 \leq n \leq k$:
		
		\vspace{+1ex}
		If $\beta(b_{k}, c_{k}, n) = r(n)$ then PA proves $[Bt(b_{k}, c_{k}, n, r(n))]$.
		
		\vspace{+1ex}
		\noindent We now define the arithmetical formula $[R(b_{k}, c_{k}, n)]$ for any $1 \leq n \leq k$ by:
		
		\vspace{+1ex}
		$[R(b_{k}, c_{k}, n) = r(n)]$ if, and only if, PA proves $[Bt(b_{k}, c_{k}, n, r(n))]$.
		
		\vspace{+1ex}
		\noindent Hence every putatively \textit{well-defined} real number $\mathbb{R}$ in the interval $0 < \mathbb{R} \leq 1$ uniquely corresponds to an algorithmically \textit{verifiable} arithmetical formula $[R(x)]$ since:
		
		\begin{quote}
			\normalsize
			For any $k$, the primitive recursivity of $\beta(b_{k}, c_{k}, n)$ yields a deterministic algorithm AL$_{(\beta, \mathbb{R}, k)}$ that can provide evidence\footnote{In the sense of \cite{Mu91} and \cite{Lob59}; see \S \ref{sec:revisiting}.} for deciding the unique value of each formula in the finite sequence $\{[R(1), R(2), \ldots, R(k)]\}$ by evidencing the truth under an \textit{evidence-based} interpretation of PA for:
			
			\begin{quote}
				\normalsize
				$[R(1) = R(b_{k}, c_{k}, 1)]$ \\
				$[R(b_{k}, c_{k}, 1) = r(1)]$
				
				\vspace{+1ex}
				$[R(2) = R(b_{k}, c_{k}, 2)]$ \\
				$[R(b_{k}, c_{k}, 2) = r(2)]$
				
				\vspace{+1ex}
				\ldots
				
				\vspace{+1ex}
				$[R(k) = R(b_{k}, c_{k}, k)]$ \\
				$[R(b_{k}, c_{k}, k) = r(k)]$.
			\end{quote}
		\end{quote}
		
		\noindent The correspondence is unique because, if $\mathbb{R}$ and $\mathbb{S}$ are two different putatively \textit{well-defined} reals in the interval $0 < \mathbb{R},\ \mathbb{S} \leq 1$, then there is always some $m$ for which:
		
		\vspace{+1ex}
		$r(m) \neq s(m)$.
		
		\vspace{+1ex}
		\noindent Hence the corresponding arithmetical formulas $[R(n)]$ and $[S(n)]$ are such that:
		
		\vspace{+1ex}
		$[R(n) = r(n)]$ for all $1 \leq n \leq m$.
		
		\vspace{+.5ex}
		$[S(n) = s(n)]$ for all $1 \leq n \leq m$.
		
		\vspace{+.5ex}
		$[R(m) \neq S(m)]$.
		
		\vspace{+1ex}
		\noindent By \cite{An16}, \S 2, Theorem 2.1\footnote{See also \cite{An21}, \S 7.G, Theorem 7.2: \textit{There are \textit{well-defined} number theoretic functions that are algorithmically \textit{verifiable} but not algorithmically \textit{computable}}.}, there is an algorithmically \textit{uncomputable} real number $\mathbb{R}$ such that the corresponding PA formula $[(\exists y)(R(x) = y)]$ is also algorithmically \textit{uncomputable}, but algorithmically \textit{verifiable}, under any \textit{evidence-based} interpretation of PA over $\mathbb{N}$.
		
		\vspace{+1ex}
		\noindent The theorem follows. \hfill $\Box$
	\end{proof}
	
	We conclude that if we were to unambiguously separate the classes P and NP as in \S \ref{sec:pnp.sprtn.prblm}, then it would follow that:
	
	\begin{corollary}[P$\neq$NP by separation]
		\label{cor:pvnp.sprtn.prf}
		If P is the class of problems that admit algorithmically \textit{computable} solutions, and NP is the class of problems that admit algorithmically \textit{verifiable} solutions, then P $\neq$ NP. \hfill $\Box$
		
		\index{P$v$NP}%
	\end{corollary}

	\subsubsection{Why the class NP is not \textit{well-defined}}
	\label{sec:np.nt.wl.dfnd}
	
	We can now see why the classical definition of NP cannot claim to be \textit{well-defined}:
	
	\begin{theorem}[NP is algorithmically \textit{verifiable}]
		\label{thm:np.alg.ver}
		If a number-theoretical formula $[F(n)]$ is in NP, then it is algorithmically \textit{verifiable}.
	\end{theorem}
	
	\begin{proof}
		By the accepted-as-standard definition of NP\footnote{In \cite{Cook}.}, if $[F(n)]$ is in NP, then it is classically computable by a non-deterministic Turing machine, say NDTM, in polynomial time. Hence, for any \textit{specified} natural number $k$, NDTM computes the sequence $\{[F(1), F(2), \ldots, F(k)]\}$ in polynomial time. By G\"{o}del's $\beta$-function (see \S \ref{gdl.beta.fn}), we can define a primitive recursive function $[G_{k}(n)]$ such that $[G_{k}(i) = F(i)]$ for all $1 \leq i \leq k$. By \S \ref{sec:revisiting}, Definition \ref{sec:2.def.2}, $[G_{k}(n)]$ is algorithmically \textit{computable}. The theorem follows. \hfill $\Box$	
	\end{proof}
	
	Thus, for NP to be a \textit{well-defined} class we would---in view of \S \ref{sec:lgcl.prf.p.neq.np}, Theorem \ref{thm:sat.prf.pvnp} (SAT is not in P or NP)---need to prove, conversely, that if $[F(n)]$ is algorithmically \textit{verifiable}, then it must be classically computable by a non-deterministic Turing machine \textit{in polynomial time}.
	
	\vspace{+1ex}
	Prima facie, such a proof is neither obvious, nor intuitively plausible from the \textit{evidence-based} perspective of this investigation, as there seems to be no conceivable reason why---even in principle since \textit{evidence-based} reasoning treats a formula that is not algorithmically \textit{verifiable} as \textit{ill-defined}\footnote{See \cite{An21}, \S 7.F: \textit{Well-definedness}.}---\textit{every} \textit{well-defined} number-theoretic formula \textit{must}, necessarily, be classically computable by a non-deterministic Turing machine \textit{in polynomial time}.

	\vspace{+1ex}
	However, such a putative proof seems precisely what is implicitly appealed to in the 2019 claim \cite{AAB19}\footnote{Already cogently challenged on the basis of competing experimental data by competing industry researchers, and on the basis of theoretical considerations (see \cite{An21}, \S 23: \textit{The significance of evidence-based reasoning for quantum computing}).} by a 78-member team of researchers, to have successfully reached the threshold of quantum supremacy by \href{https://static-content.springer.com/esm/art\%3A10.1038\%2Fs41586-019-1666-5/MediaObjects/41586\_2019\_1666\_MOESM1\_ESM.pdf}{\textit{building}}\footnote{Structured, prima facie, essentially as in Deutsch, \cite{Deu85} (see also Fiske, \cite{Fi19}; and \S \cite{An21}, 20.G: \textit{On the Collatz conjecture and similar, open, arithmetical problems}).} at Google AI Quantum, Mountain View, California, USA, a:
	
	\begin{quote}
		\footnotesize
		`\ldots high-fidelity processor capable of running quantum algorithms in an exponentially large computational space \ldots' \\ \textit{\tiny{\ldots Arute, Arya, Babbush, \textit{et al}: \cite{AAB19}, \S The future.}}
	\end{quote}

	\subsubsection{An \textit{evidence-based} separation of computational complexity}
	\label{sec:eb.sep.cmp.cmplxty}
	
	The preceding argumentation of \S \ref{sec:pnp.sprtn.prblm} suggests that a more natural separation of computational complexity---that takes into account aspects of human mathematical cognition which, even if admitted as capable of being \textit{evidenced} in what Markus Pantsar terms as `preformal mathematics' in \cite{Pan09}, may not be formalisable mathematically in terms of provable formulas---could be to distinguish between:
	
	\begin{itemize}
		\item[(i)] all that is algorithmically \textit{computable} by a deterministic Turing machine in \textit{polynomial} time; which does not include FACTORISATION\footnote{By \cite{An21}, \S 21.A.f, Theorem 21.16: \textit{P$\neq$NP since there are integers $n$ such that no deterministic algorithm can compute a prime factor of $n$ in polynomial-time}.} and SAT\footnote{By \S \ref{sec:lgcl.prf.p.neq.np}, Theorem \ref{thm:sat.prf.pvnp}.};
		
		\item[(ii)] all that is algorithmically \textit{computable} by a deterministic Turing machine in \textit{exponential} time; which includes FACTORISATION but does not include SAT;
		
		\item[(iii)] all that encompasses evidencing algorithmically \textit{verifiable} truths by meta-reasoning in \textit{finite} time; which includes SAT, since a human intelligence can evidence the algorithmically \textit{verifiable} truth of the G\"{o}del sentence $R(x)$ (by \S \ref{sec:7}, Corollary \ref{sec:6.3.cor.1.1})\footnote{See also \cite{An16}, Corollary 8.3: \textit{In any model of PA, G\"{o}del's arithmetical formula $[R(x)]$ interprets as an algorithmically \textit{verifiable}, but not algorithmically \textit{computable}, tautology over $\mathbb{N}$}.} by meta-reasoning in \textit{finite} time; reasoning which, however\footnote{By \cite{An21}, \S 20.E: \textit{Are you a man or a machine: A Definitive Turing Test}.}, is not admitted by any mechanistic intelligence whose architecture admits the classical Church-Turing thesis.
	\end{itemize}
	
	We conclude by noting that, prima facie, referencing a Turing Test\footnote{Such as \cite{An21}, \S 20.E, Query 21: \textit{Can you prove that, for any well-defined numeral $[n]$, G\"{o}del's arithmetic formula $[R(n)]$ is a theorem in the first-order Peano Arithmetic PA, where $[R(x)]$ is defined by its G\"{o}del number $r$ in eqn.12, and $[(\forall x)R(x)]$ is defined by its G\"{o}del number $17Gen\ r$ in eqn.13, on p.25 of \cite{Go31}? Answer only either `Yes' or `No'}.}, in para (iii) above, could admit aspects of human mathematical cognition such as those addressed by Pantsar in \cite{Pan19}; doing justice to these considerations, however, lies outside the scope and competence of this \textit{evidence-based} investigation:
	
	\begin{quote}
		\footnotesize
		``\ldots In a purely computational-level approach it is natural to assume that human competence can be modeled by optimal algorithms for solving mathematical problems, rather than studying empirically what kind of problem solving algorithms actual human reasoners use.
		
		\vspace{+1ex}
		While this computational-level approach has clear advantages, I submit that there should be limits to how strong and wide the application of the a priori computational methodology should be. As fruitful as the computational complexity paradigm may be, we should not dismiss the possibility that human competence in mathematical problem solving may indeed differ in important and systematic ways from the optimal algorithms studied in the computational complexity approach. In the rest of this paper, I will argue that by including considerations on the algorithmic level, we can get a more informative framework for studying the actual human problem solving capacity. Furthermore, I will show that the algorithmic-level approach does not move the discussion from competence to performance. Instead, we get a theoretical framework that is better-equipped for explaining human competence by including considerations of the algorithms that are cognitively optimal for human reasoners." \\ \textit{\tiny{\ldots Pantsar: \cite{Pan19}, \S 5, Complexity Within P and the Computational Paradigm.}}
	\end{quote}

	\section{Appendix: G\"{o}del's $\beta$-function}
	\label{gdl.beta.fn}
	
	\noindent We note that G\"{o}del's $\beta$-function is defined as (\cite{Me64}, p.131):
	
	\begin{itemize}
		\item[ ] $\beta (x_{1}, x_{2}, x_{3}) = rm(1+(x_{3}+ 1) \star x_{2}, x_{1})$
	\end{itemize}
	
	\noindent where $rm(x_{1}, x_{2})$ denotes the remainder obtained on dividing $x_{2}$ by $x_{1}$.
	
	\vspace{+1ex}
	\noindent We also note that:
	
	\begin{lemma}
		\label{sec:ch.2.2.2.lem.1}
		For any non-terminating sequence of values $f(0), f(1), \ldots$, we can construct natural numbers $b_{k},\ c_{k}$ such that:
		
		\begin{itemize}
			\item[(i)] $j_{k} = max(k, f(0), f(1), \ldots, f(k))$;
			
			\item[(ii)] $c_{k} = j_{k}$!;
			
			\item[(iii)] $\beta(b_{k}, c_{k}, i) = f(i)$ for $0 \leq i \leq k$.
		\end{itemize}
	\end{lemma}
	
	\noindent \textit{Proof} This is a standard result (\cite{Me64}, p.131, Proposition 3.22). \hfill $\Box$
	
	\vspace{+1ex}
	\noindent Now we have the standard definition (\cite{Me64}, p.118):
	
	\begin{definition}
		\label{sec:ch.2.2.1.def.1}
		A number-theoretic function $f(x_{1}, \ldots, x_{n})$ is said to be representable in the first order Peano Arithmetic PA if, and only if, there is a PA formula $[F(x_{1}, \dots, x_{n+1})]$ with the free variables $[x_{1}, \ldots, x_{n+1}]$, such that, for any specified natural numbers $k_{1}, \ldots, k_{n+1}$:
		
		\begin{itemize}
			\item[(i)] if $f(k_{1}, \ldots, k_{n}) = k_{n+1}$ then PA proves: $[F(k_{1}, \ldots, k_{n}, k_{n+1})]$;
			
			\item[(ii)] PA proves: $[(\exists_{1} x_{n+1})F(k_{1}, \ldots, k_{n}, x_{n+1})]$.
		\end{itemize}
		
		The function $f(x_{1}, \ldots, x_{n})$ is said to be strongly representable in PA if we further have that:
		
		\begin{itemize}
			\item[(iii)] PA proves: $[(\exists_{1} x_{n+1})F(x_{1}, \ldots, x_{n}, x_{n+1})]$. \hfill $\Box$
		\end{itemize}
	\end{definition}
	
	\noindent We also have that:
	
	\begin{lemma}
		\label{sec:ch.2.2.1.lem.1}
		$\beta(x_{1}, x_{2}, x_{3})$ is strongly represented in PA by $[Bt(x_{1}, x_{2}, x_{3}, x_{4})]$, which is defined as follows: 
		
		\begin{itemize}
			\item[ ] $[(\exists w)(x_{1} = ((1 + (x_{3} + 1)\star x_{2}) \star w + x_{4}) \wedge (x_{4} < 1 + (x_{3} + 1) \star x_{2}))]$.
		\end{itemize}
	\end{lemma}
	
	\textit{Proof} This is a standard result (\cite{Me64}, p.131, proposition 3.21). \hfill $\Box$

\end{document}